\documentclass[11pt,a4paper]{amsart}		

\usepackage[applemac]{inputenc}					

\usepackage{lmodern}						
\usepackage[english]{babel}					
\usepackage[weather]{ifsym}	
\usepackage{mathrsfs}				

\usepackage{hyperref,enumitem,soul}

\hypersetup{
	colorlinks,
	linkcolor={red!50!black},
	citecolor={blue!50!black},
	urlcolor={blue!80!black}
}

\usepackage{tikz}
\usetikzlibrary{calc, svg.path, patterns}
\usepackage{todonotes}

\usepackage{amsmath,amsfonts,amssymb,amsthm}

\usepackage[left=3cm,right=3cm,top=3cm,bottom=3cm]{geometry}

\usepackage{mathtools}						

\def\refer#1{~\ref{#1}}
\def\refeq#1{~(\ref{#1})}
\def\ccite#1{~\cite{#1}}

\def\inte#1{
\displaystyle\mathop{#1\kern0pt}^\circ }

\def\sumetage#1#2{\sum_{\substack{{#1}\\{#2}}}}

\let\b=\beta

\let\lam=\lambda

\let\wt=\widetilde

\def\cB{{\mathcal B}}

\def\cG{{\mathcal G}}

\def\cI{{\mathcal I}}
\def\cJ{{\mathcal J}}

\def\cL{{\mathcal L}}

\def\cR{{\mathcal R}}
\def\cS{{\mathcal S}}

\def\cU{{\mathcal U}}

\def\cW{{\mathcal W}}

\def\virgp{\raise 2pt\hbox{,}}
\def\cdotpv{\raise 2pt\hbox{;}}
\def\eqdefa{\buildrel\hbox{\footnotesize def}\over =}

\def\im {\mathop{\rm Im}\nolimits}

\def\C{\mathop{\mathbb C\kern 0pt}\nolimits}
\def\DD{\mathop{\mathbb D\kern 0pt}\nolimits}
\def\EE{\mathop{{\mathbb E \kern 0pt}}\nolimits}
\def\K{\mathop{\mathbb K\kern 0pt}\nolimits}
\def\N{\mathop{\mathbb N\kern 0pt}\nolimits}
\def\Q{\mathop{\mathbb Q\kern 0pt}\nolimits}
\def\R{{\mathop{\mathbb R\kern 0pt}\nolimits}}
\def\SS{\mathop{\mathbb S\kern 0pt}\nolimits}
\def\ZZ{\mathop{\mathbb Z\kern 0pt}\nolimits}
\def\TT{\mathop{\mathbb T\kern 0pt}\nolimits}
\def\P{\mathop{\mathbb P\kern 0pt}\nolimits}
\def \H{{\mathop {\mathbb H\kern 0pt}\nolimits}}
\newcommand{\ds}{\displaystyle}

\DeclareMathOperator{\tr}{Tr}
\DeclareMathOperator{\re}{Re}

\newcommand{\beq}{\begin{equation}}
\newcommand{\eeq}{\end{equation}}
\newcommand{\ben}{\begin{eqnarray}}
\newcommand{\een}{\end{eqnarray}}
\newcommand{\beno}{\begin{eqnarray*}}
\newcommand{\eeno}{\end{eqnarray*}}
\newcommand{\bqs}{\begin{equation*}}
\newcommand{\eqs}{\end{equation*}}
\newcommand{\andf}{\quad\hbox{and}\quad}
\newcommand{\with}{\quad\hbox{with}\quad}

\def\equivH#1 {\buildrel\hbox{\tiny {$#1$}}\over \equiv}
\def\simH#1 {\buildrel\hbox{\footnotesize {$#1$}}\over \sim}

\newtheorem{definition}{Definition}[section]
\newtheorem{theorem}{Theorem} 
\newtheorem{lemma}{Lemma}[section]
\newtheorem{remark}{Remark}[section]
\newtheorem{cor}{Corollary}[section]
\newtheorem{proposition}{Proposition}[section]
\numberwithin{equation}{section}

\begin{document}
\title[The derivative nonlinear Schr\"odinger equation]
{Global well-posedness for the derivative nonlinear Schr\"odinger equation}

\author[H. Bahouri]{Hajer Bahouri}
\address[H. Bahouri]
{CNRS  \&  Sorbonne Universit\'e  \\
 Laboratoire Jacques-Louis Lions (LJLL) UMR  7598 \\
4, Place Jussieu\\
75005 Paris, France.}
\email{hajer.bahouri@ljll.math.upmc.fr}
\author[G. Perelman]{Galina Perelman}
\address[G. Perelman]%
{Laboratoire d'Analyse et de Math{\'e}matiques Appliqu{\'e}es UMR 8050 \\
Universit\'e Paris-Est  Cr{\'e}teil\\
61, avenue du G{\'e}n{\'e}ral de Gaulle\\
94010 Cr{\'e}teil Cedex, France}
\email{galina.perelman@u-pec.fr}

\date{\today}

\begin{abstract} This paper is dedicated to the study of the derivative nonlinear Schr\"odinger equation on the real line. The local well-posedness of this equation in the Sobolev spaces $H^s(\R)$  is well understood since a couple of decades, while the global well-posedness is not completely settled.  For the latter issue, the best known results up-to-date concern either Cauchy data   in $H^{\frac 1 2}(\R)$ with mass strictly less than $4\pi$ or general initial conditions in the weighted   Sobolev space  $H^{2, 2}(\R)$.  In this article, we prove that the derivative nonlinear Schr\"odinger equation is globally well-posed for general Cauchy data in $H^{\frac 1 2}(\R)$ and that furthermore the $H^{\frac12}$ norm of the solutions remains globally bounded in time. 
One should recall that for $H^{s}(\R)$, with $s < 1 / 2 $,  the associated Cauchy problem is ill-posed in the  
sense that uniform continuity with respect to the initial data fails. Thus, our result closes the discussion in the setting of the Sobolev spaces $H^s$. The proof is achieved by combining the profile decomposition techniques with the integrability structure of the equation.

\end{abstract}

\maketitle

\tableofcontents


\vskip 0.2cm


\section{Introduction}\label {intro}
\setcounter{equation}{0}
 This paper aims to investigating global well-posedness  for the derivative nonlinear Schr\"odinger equation (DNLS) on the real line: 
\begin{equation}\label {eq:DNLS} 
iu_t +u_{xx} =  \pm  i \partial_x(|u|^2 u), \quad x\in \R,
\end{equation}
with initial conditions
\begin{equation}\label{id}u|_{t=0} = u_0 \in H^{s}(\R), \,\,\, s\geq 1/2.
\end{equation}

The transformation $u(t,x)\rightarrow u(t,-x)$ maps solutions of\refeq{eq:DNLS}
 with sign $-$ to solutions of\refeq{eq:DNLS}
 with sign $+\cdot$   In what follows,    we shall   fix  the sign $-$ in \refeq{eq:DNLS}.

The   DNLS equation was derived by Mio-Ogino-Minami-Takeda and Mjolhus in\ccite{MOMT, mjohus} for studying the one-dimensional compressible magneto-hydrodynamic  equation in the presence of the Hall effect and the propagation of circular polarized nonlinear Alfv\'en waves in  magnetized plasmas\footnote{The  DNLS equation also appears   as a model for ultrashort
optical pulses\ccite{Moses}. For an outline on physical applications of this  equation, one can  consult\ccite{Champeaux, Sulem} and the references therein.}.

The equation \eqref{eq:DNLS} is known to be completely integrable, and to admit an infinite number of conservation laws including
the conservation of mass, momentum and  energy:
\begin{equation}
\label{mass} M(u)\eqdefa\int_{\R}|u|^2  dx \, ,
\end{equation}
\begin{equation}
\label{momentum}  P(u)\eqdefa  {\rm Im} \int_{\R}\overline {u} u_x  dx+ \frac 1 2 \int_{\R}|u|^4 dx \virgp
\end{equation}
\begin{equation}
\label{energy} E(u)\eqdefa \int_{\R} \Big(|u_x |^2 -\frac 3 2 {\rm Im} (|u|^2 u \overline {u}_x) + \frac 1 2  |u|^6\Big)dx \cdotp \end{equation}

The problem of local and eventually global  well-posedness for the DNLS equation has received a lot of attention over the last twenty years.
The local well-posedness   is  fully understood in the scale of the Sobolev spaces $H^s(\R)$:
combining a gauge
transformation with the Fourier restriction method,  Takaoka
 proved in\ccite{HT}  that  the corresponding  Cauchy problem is locally well-posed 
in~$H^s(\R)$ 
for~$s \geq   1 / 2$,  improving   the earlier $H^1(\R)$-result of Hayashi and Ozawa \ccite{Ozawa0}.
The Takaoka's result  is optimal accordingly to the works\ccite{Biagioni, HT2} where it was shown that  the Cauchy problem\refeq{eq:DNLS}-\eqref{id}  
 is ill-posed in~$H^s(\R)$ for $s <    1 / 2$, in the sense that uniform continuity with respect to the initial conditions fails.   
 Note that DNLS
 is $L^2$-critical being  invariant under the scaling: 
\begin{equation}
\label{scaling}
   u(t,x)\longrightarrow u_\mu(t,x) \eqdefa \sqrt{\mu} u(\mu^2 t, \mu x), \quad \mu >0 \, .
\end{equation}
The 1/2 derivative gap in the local well-posedness can be closed by leaving the $H^s(\R)$-scale and considering 
more general functional spaces, see for instance\ccite{AG}  and the references therein.

Concerning the question of global well-posedness,  Hayashi and Ozawa \ccite{HO} proved the global existence for $H^1$ solutions with initial data satisfying $\|u_0\|_{L^2(\R)}< \sqrt{2\pi}.$ By the sharp Gagliardo-Nirenberg inequality  \ccite{Weinstein}
\begin{equation}
\label{sharpin0}\|f\|^6_{L^6(\R)} \leq \frac{4}{\pi^2}\|f\|^{4}_{L^2(\R)} \|f_x\|^{2}_{L^2(\R)} \, , \, \, \forall f \in  H^1(\R)\, ,\end{equation}
this smallness assumption allows to control the
$H^1$-norm of the solution by its energy and  the mass\footnote{Note  that, under the gauge transformation 
$\ds u=ve^{-\frac{3i}{4}\int_{-\infty}^x|v(y)|^2}$, the energy \eqref{energy} reduces to the energy of the focusing quintic nonlinear Schr\"odinger equation
$iv_t +v_{xx} = -\frac3{16}|v|^4v$.}. This result was extended to $H^s$ data with $s>1/2$ by Colliander-Keel-Staffilani-Takaoka-Tao\ccite{CKSTT}. More recently, Wu\ccite{W} and Guo-Wu \ccite{Guo} increased the upper bound $\sqrt{2\pi}$
to $\sqrt{4\pi}$ respectively for $H^1$ and $H^{\frac 1 2}$ solutions. In the  $H^1$ setting, the result follows from the mass, momentum and energy conservation combined with the following Gagliardo-Nirenberg inequality\ccite{G} \begin{equation}
\label{sharpin}\|f\|_{L^6(\R)} \leq C_{GN} \|f\|^{\frac 8 9}_{L^4(\R)} \|f_x\|^{\frac 1 9}_{L^2(\R)} \, , \, \, \forall f \in L^4(\R)\cap \dot H^1(\R)\, ,\end{equation}
where the optimal constant $C_{GN}$ is given by
$C_{GN} 
= 3^{\frac 1 6} (2\pi)^{- \frac 1 9}$. 
The extension to $H^{\frac 1 2}$-solutions was achieved 
by using  the I-method.

The both bounds $\sqrt{2\pi}$ and $\sqrt{4\pi}$ are related  to the $L^2$-norm of solitary wave solutions of the DNLS equation.
These solutions can be written in the explicit form:
\begin{equation}
\label{particularbright} \ds u_{E,c}(t,x) = e^{ i \omega t+ i\frac c 2 x - \frac 3 4 i  \int^{x-ct}_{-\infty} |\varphi_{E, c}(s)|^2 ds } \varphi_{E, c}(x-ct ) \, ,\end{equation}
where  $E>0$, $c \in  \R$, $\omega=E-\frac  {c^2} 4 $ and
$$  \varphi_{E, c}(y)= \frac {2 \sqrt{2 E} } {(c^2+4E)^ {\frac 1 4}}\, \frac 1 {\sqrt{\cosh (2 \sqrt{E} y)- \frac c {\sqrt{c^2+4 E}}}} \virgp$$
is the unique positive even exponentially decaying solution of
\begin{equation}
\label{eqsolitons} - \varphi_{yy} + E \varphi+ \frac {c} 2 \varphi^3- \frac 3 {16}   \varphi^5= 0\, \cdotp\end{equation}
These solutions are usually referred to  as  the bright solitons\footnote{Their  orbital stability   was studied in\ccite{Colin, Guo1, KW}.}.
In the limiting case $E=0$, $c>0$, the profile $\varphi_{E,c}$ reduces to
$$\varphi_{0, c}(y)=\frac{2\sqrt{c}}{\sqrt{1+c^2y^2}}\virgp$$
solving
\begin{equation}
\label{eqsolitons} - \varphi_{yy} +  \frac {c} 2 \varphi^3- \frac 3 {16}   \varphi^5= 0\, \cdotp\end{equation}
The corresponding solution $u_{0, c}$ is  called the algebraic soliton. 
One can easily check that \begin{equation}
\label{massbright}  M(u_{E,c})= 8 \arctan \sqrt{\frac {\sqrt{c^2+4 E}+c} {\sqrt{c^2+4 E}-c}} \leq 4\pi\,  \cdotp\end{equation}
The values $2\pi$ and $4\pi$ correspond  to the mass of the static bright solitons $u_{E,0}$ and the algebraic solitons 
$u_{0,c}$, their profiles $ \varphi_{E,0}$ and $\varphi_{0,c}$ being the extremals\footnote{the only extremals up to the symmetries of the DNLS equation.} of the Gagiardo-Nirenberg inequalities \eqref{sharpin0}
and \eqref{sharpin} respectively.

\smallskip There is also a number of works \ccite{Sulem1, Sulem, Sulem0,PSS, PSS2} where the global well-posedness of the DNLS equation was studied by means of the inverse scattering 
techniques.
The corresponding results get rid of the smallness assumption on the mass, but require more regularity and decay on the initial data.
The most definite result is due to 
Jenkins, Liu, Perry and Sulem who proved in \ccite{Sulem1} that the Cauchy problem for the DNLS equation is globally well-posed for any 
initial  data $u_0$ in~$H^{2, 2}(\R)=\big\{f  \in H^2(\R)  : x^2 f  \in L^2(\R) \big\}\cdotp$ Let us also mention the work of Pelinovsky-Saalmann-Shimabukuro\ccite{PSS} that gives the global well-posedness for generic initial data in~$H^2(\R)\cap H^{1,1}(\R)$.

\medskip
The purpose of this paper is to prove 
the global well-posedness  of the  DNLS equation  for general   initial data in $H^{s}$, $s\geq 1/2$. Our main result    is the following:
 \begin{theorem}
\label{Mainth}
{\sl For any initial data $u_0\in H^{\frac 1 2}(\R)$, the Cauchy problem \eqref{eq:DNLS}-\eqref{id} is globally well-posed,
and the corresponding solution $u$ satisfies
\begin{equation}\label{b}
\sup\limits_{t\in \R}\|u(t)\|_{H^{\frac 1 2}(\R)}<+\infty.\end{equation}}
\end{theorem}
 \begin{remark} 
{\sl \label{commentth1}
Combining the conservation laws with the $H^{\frac 1 2}$ bound \eqref{b}, it is possible to show that  if the initial datum is in $H^s(\R)$
for some $s>1/2$, 
then the  
$H^s$-norm of the solution remains globally bounded in time as well.
We are planning to address this issue in a subsequent paper.}
\end{remark}

 The proof of  Theorem \ref{Mainth}  relies heavily on the complete integrability  of the DNLS equation, but avoids a direct use of the inverse scattering transform that requires a localization of initial data and breaks down for the solutions we are considering in this article. Instead, we exploit  as much as possible the conservation quantities, namely 
 the conservation of the transmission coefficient of the corresponding spectral problem that remains well defined for $L^2$ data, as soon as we stay away from the spectrum,
the property that has been already extensively used  in the works of  Killip-Visan, Killip-Visan-Zhang  and Koch-Tataru \ccite{killip0, killip, KT}  on the low regularity solutions of the cubic NLS and KdV equations
on the real line\footnote{See also the recent paper of Klaus-Schippa\ccite{ks},  where this property  has been used to 
obtain   low regularity a priori estimates for small mass solutions of the DNLS equation. }.

\medskip The structure of the paper  is as follows. Section~\ref{preliminarystatement2} is devoted to the preliminary  results related to the integrable structure of the DNLS equation, that will be needed in the proof of Theorem\refer{Mainth}. In the first subsection, we describe the zero curvature formulation of the DNLS equation and introduce the main elements of the
inverse scattering analysis following the founding paper of Kaup-Newell\ccite{KN}. In the second subsection,   we study
 the properties
of the corresponding spectral problem in the case of $H^{\frac 1 2}$ potentials.  The last subsection is devoted to the B\"acklund transformation and its basic properties. 
In Section\refer{proofmainth}, we establish Theorem\refer{Mainth}. We argue by contradiction.
In the subsection \refer{preliminarystatement1}, combining the profile decomposition techniques with the integrability structure of the equation,  we show that if the theorem fails, it  would imply the existence of solutions of a very special structure. We then use the B\"acklund trasformation to show that in fact, such solutions cannot exist. This is done in the subsection\refer{End}.
There are  also two  appendices:  the  first one contains  a short introduction to the regularized determinants that play an important role in Section~\ref{preliminarystatement2}, and   the second  one is dedicated to the proof of a technical estimate related to the B\"acklund transformation. 
\smallskip

Throughout this article, we shall use the following convention
for the Fourier transform: 
\begin{equation*}
\hat f(\xi)= \frac 1 {\sqrt {2 \pi}} \int_\R e^ {-i x \xi}f(x) dx.
\end{equation*}
We shall  designate by~$\mathscr{C}_n$ the set of bounded operators~$A$ on~$L^2(\R, \C^2)$ such that $|A|^n$ is of trace-class, endowed with the norm $\|A\|_n \eqdefa \big[{ \rm Tr}  \big(|A|^n\big)\big] ^{\frac 1 n}$.

\smallskip  Finally, we mention that the letter $C$ will be used to denote  universal constants
which may vary from line to line. If we need the implied constant to depend on parameters, we shall indicate this by subscripts.
We  also use the notation $A\lesssim B$ to
denote the bound of the form $A\leq C B$,   and $A \lesssim_\alpha B$ 
for $A\leq C_\alpha B$, where $C_\alpha$ depends only  on $\alpha$.
For simplicity, we shall  still denote by~$(u_n)$ any
subsequence of~$(u_n)$.

\section{Preliminary results in connexion with the integrability structure of  the DNLS equation}\label {preliminarystatement2} 
\subsection{An overview  of the scattering transform}
In this subsection we recall briefly some basic facts about the inverse scattering transform for the DNLS equation, limiting ourselves to the 
case of Schwartz class solutions. The details can be found in\ccite{Ab, Sulem1, Sulem, Sulem0, KN, Lee,  PSS2, T2}.

As was shown by Kaup-Newell\ccite{KN}, the DNLS equation arises as compatibility condition of the following linear system
 \begin{equation}
\label{system}\begin{array}{c}
\partial_x \psi= \cU(\lam)  \psi,\\
\partial_t \psi= \Upsilon(\lam)  \psi\,,
\end{array}
 \end{equation} 
with
\begin{eqnarray*}\cU(\lam) &= &-i \sigma_3(\lam^2 + i\lam U), \quad U=\left(
\begin{array}{ccccccccc}
0 &u \\
\overline {u} &0
\end{array}
\right), \\ \Upsilon(\lam) &= &-i (2\lam^4- \lam^2\, |u|^2) \sigma_3 + \left(
\begin{array}{ccccccccc}
0 & 2 \lam^3u - \lam |u|^2 u + i \lam u_x\\ 

-2 \lam^3\overline {u}+  \lam |u|^2 \overline {u} + i \lam \overline {u_x}&0
\end{array}
\right) \, ,\end{eqnarray*}
where  $\lam\in\C$  is a $(t,x)$-independent spectral parameter, $\psi$  is a $\C^2$-valued function of $(t, x, \lam)$, and~$\sigma_3$   the Pauli matrix given by 
$\sigma_3= 
\left(
\begin{array}{ccccccccc}
1 &0 \\
0 &-1
\end{array}
\right)\cdot$  
Namely, $u$ satisfies the DNLS equation if and only if 
$$\frac{\partial \cU}{\partial t}-\frac{\partial \Upsilon}{\partial x}+ [\cU, \Upsilon]=0,$$
which is referred to  in the literature as the zero curvature representation of DNLS.

\smallskip
The scattering transform associated with the DNLS  equation is defined via the first equation of \eqref{system} that we rewrite in the form
\begin{equation}\label{sp}
L_u(\lambda)\psi=0,
\end{equation}
with $L_u(\lambda)=i\sigma_3\partial_x-\lambda^2-i\lambda U$. 
Given $u\in \mathcal{S}(\R)$ (for the moment we ignore the time dependence), for any $\lambda \in \C$ with $\im \lambda^2\geq0$, 
there are unique solutions  $\psi_1^-(x, \lambda)$, $\psi_2^+(x, \lambda)$ to \eqref{sp}, the so-called Jost solutions,  satisfying
\begin{eqnarray*}
\psi_1^-(x, \lambda)&=& e^{ -i \lambda^2 x} \left[ \left(
\begin{array}{ccccccccc}
1  \\
0 
\end{array}
\right) + \circ (1)\right], \quad \mbox{as} \quad x \to - \infty \, ,\\
\psi_2^+(x, \lambda) &= & e^{ i \lambda^2 x} \, \, \, \left[ \left(
\begin{array}{ccccccccc}
0  \\
1 
\end{array}
\right) + \circ (1)\right], \quad \mbox{as} \quad x \to + \infty\,.
\end{eqnarray*} 
The solutions $\psi_1^-$, $\psi_2^+$ are holomorphic functions of $\lambda$ on $\Omega_+=\{\lambda\in \C: \,\ \im \lambda^2>0\}$,
$C^\infty$ up to the boundary.
Similarly, for $\lambda \in \C$ with $\im \lambda^2\leq 0$, there are 
unique solutions  $\psi_2^-(x, \lambda)$, $\psi_1^+(x, \lambda)$ to \eqref{sp} satisfying
\begin{eqnarray*}
\psi_2^-(x, \lambda)&=& e^{ i \lambda^2 x} \, \, \,\left[ \left(
\begin{array}{ccccccccc}
0  \\
1
\end{array}
\right) + \circ (1)\right], \quad \mbox{as} \quad x \to - \infty \, ,\\
\psi_1^+(x, \lambda) &= & e^{- i \lambda^2 x} \left[ \left(
\begin{array}{ccccccccc}
1  \\
0
\end{array}
\right) + \circ (1)\right], \quad \mbox{as} \quad x \to + \infty\,.
\end{eqnarray*} 
For $\lambda\in \R\cup i\R$,  this gives two pairs of linearly independent solutions: $\psi_1^-, \psi_2^-$ and $\psi_1^+, \psi_2^+$.
We denote the corresponding transfer matrix by $t_u(\lambda)=\begin{pmatrix} a_u(\lambda) & c_u(\lambda)\\b_u(\lambda) & d_u(\lambda)\end{pmatrix} $:
$$\begin{pmatrix} \psi_1^-(x,\lambda) &\psi_2^-(x,\lambda)\end{pmatrix}=\begin{pmatrix} \psi_1^+(x,\lambda) &\psi_2^+(x,\lambda)\end{pmatrix}t_u(\lambda).$$
The functions $\frac{1}{a_u}$, $\frac1{d_u}$ and $\frac{b_u}{a_u}$, $\frac{c_u}{d_u}$ are called transmission and reflection coefficients respectively.
Thanks to the symmetry relations
\begin{equation}\label{sym}\begin{split}
&\psi_1^-(x,\lambda)= \, \, \, \sigma_3\psi_1^-(x,-\lambda), \quad  \,  \psi_2^+(x, \lambda)=-\sigma_3\psi_2^+(x,-\lambda),\\
&\psi_1^-(x, \lambda)=-\sigma_1\sigma_3\overline{\psi_2^-(x, \bar \lambda)}, \quad \psi_2^+(x,\lambda)=\sigma_1\sigma_3\overline{\psi_1^+(x, \bar \lambda)},
\end{split}
\end{equation}
where 
$\sigma_1= 
\left(
\begin{array}{ccccccccc}
0 &1 \\
1 &0
\end{array}
\right)$,
one has, for all $\lambda\in \R\cup i \R$, 
\begin{eqnarray*} a_u(\lambda) &= & a_u(-\lambda), \,\,\, a_u(\lambda)=\overline{d_u(\bar\lambda)},  \\ b_u(-\lambda) &= &-b_u(\lambda), \,\,\,  c_u(\lambda)=-\overline{b_u(\bar \lambda)}.\end{eqnarray*} 
Since $\det t_u(\lambda)=1$, the above relations imply that
\begin{eqnarray} \label{identab}|a_u(\lambda)|^2+|b_u(\lambda)|^2 &=&1, \quad \forall\,\, \lambda\in \R, \\
\label{identabim} |a_u(\lambda)|^2-|b_u(\lambda)|^2 &=&1, \quad \forall\,\, \lambda\in i\R.\end{eqnarray} 
Observe also that $a_u(0)=1$. 

\smallskip The function  $a_u$ extends analytically to $\Omega_+$ since it can be expressed through the Wronskian of $\psi_1^-$ and $\psi_2^+$:
\begin{equation}\label{w}
a_u(\lambda)=\det(\psi_1^-(x,\lambda), \psi_2^+(x,\lambda)).
\end{equation}
The relation \eqref{w} also shows that the zeros of $a_u$ in $\Omega_+$ coincide with the values of $\lambda$ for which 
the system\refeq{sp} has a non trivial $L^2$ solution. In this case, we say that $\lambda$ is an eigenvalue of the spectral problem
\eqref{sp} (or of the operator pencil $L_u(\lambda)$).  These eigenvalues give rise to the bright solitons. For one-soliton solutions\refeq{particularbright},  one has
$$ \ds a_{u_{E, c}}(\lam)= e^{ -\frac {i}  {2}   \|u_{E,c}\|^2_{L^2(\R)}}\frac {\lam^2 -\zeta_{E, c}}  {\lam^2 - \overline \zeta_{E, c}}
\quad {\rm with}\quad \zeta_{E, c}= - \frac c 4 + i \frac {\sqrt E} 2\cdotp$$

To determine the behavior of $a_u$ at infinity, it is convenient to transform the Kaup-Newell
spectral problem \eqref{sp} into a more "familiar" Zakharov-Shabat spectral problem (the spectral problem associated with the cubic NLS)
which is linear with respect to the spectral parameter. This can be done by means of the following transformation \ccite{KN, PSS}:
\begin{equation}
\label{gauge} \wt\psi(x) =
 \exp\Big(\frac 1  {2i} \sigma_3 \int^\infty_x dy |u(y)|^2\Big)
 \left(
\begin{array}{ccccccccc}
1 &0 \\
- \overline {u}(x) &  2 i\lam
\end{array}
\right)\psi (x) .\end{equation}
One can easily check that  $\psi$ is a solution of \eqref{sp} if and only if $\tilde \psi$ satisfies 
\begin{equation}\label{sp1}
i\sigma_3\partial_x\tilde \psi-Q\tilde\psi=\zeta\tilde \psi, \quad Q=   \left(
\begin{array}{ccccccccc}
0 &q  \\
r &0
\end{array}
\right), 
\quad \zeta=\lambda^2,
\end{equation}
with
\begin{equation*}
 q(x) = \frac12 u(x)\exp\Big(- i  \int^\infty_x dy |u(y)|^2\Big) , \quad r(x) =  (i  \overline {u}_x+\frac12\overline {u}|u|^2) (x)\exp\Big( i  \int^\infty_x dy |u(y)|^2\Big).\end{equation*}
The equivalence between the systems \eqref{sp} and  \eqref{sp1} shows that 
\begin{equation}\label{lim}
\lim\limits_{|\lambda|\rightarrow\infty, \, \lambda\in \overline{\Omega}_+}a_u(\lambda)=e^{-\frac{i}2\|u\|^2_{L^2(\R)}}.
\end{equation}
Furthermore, denoting $\tilde a_u(\zeta)=e^{\frac{i}2\|u\|^2_{L^2(\R)}}a_u(\sqrt{\zeta})$, one has the following asymptotic expansion as~$|\zeta|\rightarrow +\infty$, $\im \zeta\geq 0$:
\begin{equation}\label{as}
\ln \tilde a_u(\zeta)=\sum\limits_{k\geq 1} E_k(u)\zeta^{-k}.\end{equation}
The coefficients $E_k$ are polynomial with respect to  $u$ and its derivatives and can be determined  recursively.
They are homogeneous with respect to the scaling~$u(x)\longrightarrow u_\mu(x) \eqdefa \sqrt{\mu} u(\mu x), \,  \mu >0$,
since
\begin{equation}
\label{sca}
\tilde a_{u_\mu}(\zeta)=\tilde a_u\left(\frac\zeta{{\mu}}\right)\cdotp
\end{equation}
The first two among them coincide, up to a constant,  with the  momentum and energy previously introduced by\refeq{momentum}-\eqref{energy}:
 $$E_1(u)=\frac{i}4P(u), \quad E_2(u)=-\frac{i}8E(u).$$
Note that $\tilde a_u(\zeta)$ is holomorphic in the open upper half plane $\C_+$, $C^\infty$ up to the boundary and verifies, in view of\refeq{identab}-\eqref{identabim},
\begin{equation}
\label{size}|\tilde a_u(\zeta)|\geq 1 \, \,  \mbox{for} \,  \zeta <0, \, \, |\tilde a_u(\zeta)|\leq 1 \, \,  \mbox{for} \,  \zeta >0 \andf   \tilde a_u(0)=e^{\frac{i}2\|u\|^2_{L^2(\R)}}. \end{equation}    Furthermore, one can show that $|\tilde a_u(\zeta)|^2\in 1+\cS(\R)$. The analyticity of $\tilde a_u$ allows to express the functionals $E_k$ in terms of the zeros of $\tilde a_u$ in $\C_+$  and of its trace on $\R$ (by the so-called trace formulas). In the simplest case where (i) $\tilde a_u$ does not vanish\footnote{It follows from Formula \eqref{lim} that in that case, $\tilde a_u$ has 
only a finite number of zeros in $\C_+$, all of them being of finite multiplicity.} on~$\R_+$
and (ii) $\tilde a_u$ has only simple zeros~$\zeta_1, \dots, \zeta_N$ in $\C_+$, one has the formulae
\begin{equation}
\label{asympek}E_k(u)=-\frac{2i}k\sum\limits_{j=1}^N \im\zeta_j^k+\frac{i}{2\pi}\int^\infty_{-\infty}  {d\xi}  \, \xi^{k-1} \, \ln |\tilde a_u(\xi)|^2, \quad \forall k\in \N^*,\end{equation} 
that follow immediately from the representation
\begin{equation}
\label{devol} \wt a_u(\zeta) =  \prod^N_{j=1}  \Big( \frac {\zeta -\zeta_j}  {\zeta - \overline \zeta_j} \Big) \exp\Big(\frac 1 {2i\pi} \int^\infty_{-\infty} \frac {d\xi} {\xi-\zeta} \, \ln |\tilde a_u(\xi)|^2 \Big)\, ,  \, \, \forall \zeta \in \C_+    \,  \cdotp\end{equation} 

Taking into account \eqref{size} and the continuity of $\tilde a_u$ with respect to $u$, one can also deduce from\refeq{devol} that
\begin{equation}\label{mass1}
M(u)  =  4 \sum^{N}_{j=1} {\rm arg} (\zeta_j)   - \frac 1 { \pi} \int^\infty_{-\infty} \frac  {d\xi} {\xi} \, \ln |\tilde a_u(\xi)|^2, 
\end{equation}
with, all along this paper, ${\rm arg} (\zeta)\in  [0, 2 \pi[$. 

\medskip 
It is known that  the properties (i), (ii) hold generically:
the subset of Schwartz functions~$u$ verifying the hypothesis (i) and (ii), that we shall denote all along this article by~$\cS_{reg}(\R)$,  is dense in~$\cS(\R) $ \ccite{bealscoifman,  Sulem1,  Sulem,  Sulem0,  Lee}.

\medskip 
In the above discussion,  we have suppressed  the time dependence. If we now restore it, assuming that $u(t)$ is a solution of the DNLS equation, then the time dependence
 of the scattering coefficients~$a_{u(t)}(\lambda)$ and~$b_{u(t)}(\lambda)$ can be deduced from the second equation of\refeq{system}.
 By straightforward computations, one finds a particular simple linear evolution system:
 \begin{equation}\label{time}
 \partial_ta_{u(t)}(\lambda)=0, \quad \partial_tb_{u(t)}(\lambda)=-4i\lambda^4b_{u(t)}(\lambda).
 \end{equation}
 This provides a way of solving the DNLS equation as soon as the potential $u$ can be recovered from the scattering coefficients
 $a_u$, $b_u$, which can be done if $a_u$ has no zeros in $\overline{\Omega}_+$. 
 In this case, 
 one can  reconstruct the potential $u$ from the 
 reflection coefficient $\frac{b_u}{a_u}\virgp$ by solving a suitable Riemann-Hilbert problem.  This procedure can be also adapted to the  general case but 
 the set of scattering data needed
 to reconstruct the potential becomes more intricate,  see\ccite{Sulem1, Sulem, Sulem0} for the details.
 Note also that since $a_u$ is time-independent, the expansion \eqref{as} produces an infinite number of polynomial conservation laws. 
  
\smallskip  
 The use of the inverse scattering transform is restricted to the localized data:
 although the assumption  $u\in \mathcal{S}(\R)$ can be weakened (see \ccite{Sulem1, Sulem, Sulem0, PSS, PSS2}),
 even to define the scattering data~$a_u(\lambda)$,~$b_u(\lambda)$ , $\lambda\in \R\cup i\R$, one needs at least $u\in L^1(\R)$.
 A way to overcome this difficulty and to keep a trace of the complete integrability for $H^s$ solutions, is to exploit the conservation of~$a_u(\lambda)$, for $\lambda\in \Omega_+$,  that remains well defined via \eqref{w} for $u\in L^2(\R)$. As we have already  mentioned above, this idea goes back to the works
 of Killip-Visan-Zang, Killip-Visan and  Koch-Tataru \ccite{killip0, killip, KT}  on the NLS and KdV equations, and will play a crucial role in the proof of Theorem \ref{Mainth}.


\subsection{Study of the function $a_u$ for $H^{\frac12}$ potentials} \label {preliminarystatementscat} 

In this section, we perform a detailed analysis of the function $a_u$ in $\Omega_+$ for $u$ in $H^{\frac 1 2}(\R)$.
A convenient way to do it is to realize the Wronskian~\eqref{w} as a regularized Fredholm determinant.
\subsubsection{Regularized determinant realization of $a_u$}
  Consider
\begin{equation}\label{T}T_u(\lam)\eqdefa   i \lam (\cL_0-   \lam^2)^{-1} U \, ,\end{equation}
where $\cL_0= i \sigma_3 \partial_x$ and $U=\left(
\begin{array}{ccccccccc}
0 &u \\
\overline {u} &0
\end{array}
\right)$. For any $u\in L^2(\R)$, $T_u$ is an holomorphic function of $\lambda$ in~$\Omega_+$
with values in~$\mathscr{C}_2$, and 
\begin{equation}
\label{HSuseful} \|T_u(\lam)\|^2_2   =  \frac { |\lam|^2 } {{ \rm Im} (\lam^2)} \|u\|^2_{L^2(\R)}.\end{equation} 
Indeed, 
\begin{equation}
\label{formulaTu}T_u(\lam)=  i \lam  \left(
\begin{array}{ccccccccc}
0 &- (D+ \lam^2)^{-1}u  \\
(D- \lam^2)^{-1} \overline {u} & 0\end{array}
\right)  \, , \end{equation}with $\ds D= -i  \partial_x$,  and therefore
$$\|T_u(\lam)\|^2_{2}=  \frac  { |\lam|^2}  { 2\pi} \biggl[  \int_{\R^2} d p  dp' \frac {|\hat u (p')|^2  } {|p- {\lam }^2|^2} + \int_{\R^2} d p  dp' \frac {|\hat u (p')|^2  } {|p+ {\lam }^2|^2}\biggr] \virgp $$
which readily leads to\refeq{HSuseful}, by virtue of the  Cauchy's residue theorem.

\smallskip As well,  we find that the trace of $T^2_u(\lam)$ can be written explicitly as follows:
\begin{equation}
\label{trace2useful}  \begin{aligned}   \tr T^2_u(\lam)  =  2i \lam^2\int_{\R} d p \, \frac { |\hat u (p)|^2 }   { p+ 2 \lam^2}  \, \cdotp  \end{aligned}  \end{equation}  
Using the explicit kernel of the free resolvent $(\cL_0-\lambda^2)^{-1}$:
\begin{equation}\label{kernel0}
(\cL_0-\lambda^2)^{-1}(x,y)=\left\{
\begin{array}{c}
 \left(
\begin{array}{ccccccccc}
ie^{-i \lam^2 (x-y)}  &0 \\
0 & 0\end{array}
\right)    \, \, \mbox {for}  \, \,  x < y\\  \\
 \left(
\begin{array}{ccccccccc}
0 & 0 \\
0 & ie^{i \lam^2 (x-y)}\end{array}
\right)     \, \, \mbox {for}  \, \,  x > y\, ,
\end{array}
\right. 
\end{equation}
one can also easily check  that 
there exists a positive constant $C$  such that, for any $p\geq 2$, there holds
\begin{equation}
\label{eq:controlnorml0}  \|T_u(\lam)\| \eqdefa \|T_u(\lam)\|_{\mathscr{L}(L^2, L^2)} \leq C\frac {  |\lam|\,\|u\|_{L^p(\R)}} {( { \rm Im} (\lam^2))^{1- \frac 1 p}}\virgp   \quad \forall \lambda\in\Omega_+, \,\, u\in L^p(\R). \end{equation}

The key point   will be the fact that the function $a_u$ given by \eqref{w} can be expressed in terms 
of~$T_u$ as follows\footnote{See Appendix\refer{basicdeterminants} for the definition of the regularized determinants ${\rm det}_n$ and their basic properties.}:
\begin{equation}
\label{awithdet}  a_u(\lam)= {\rm det}_2 ({\rm I}-T_u(\lam)), \quad \forall \,\lambda\in \Omega_+, \,\, u\in L^2(\R).
\end{equation}
 We also define
 \begin{equation}
\label{awithdet4}  a^{(4)}_u(\lam) \eqdefa  {\rm det}_4 ({\rm I}-T_u(\lam)) \, .\end{equation}
Similarly to $a_u$, the function $a^{(4)}_u $ is holomorphic  on $\Omega_+$, for any $u\in L^2(\R)$.
Since the matrix~$U$ is  anti-diagonal,  the  identity\refeq{dety1} implies that the two functions $a_u$ and $a^{(4)}_u$ are connected by the following 
relation:
\begin{equation}
\label{link24}  a^{(4)}_u(\lam)=  a_u(\lam) \exp \Big(  \frac {\tr T^2_u(\lam)} {2}\Big)\,   \cdotp\end{equation}

Below we collect some general bounds on the functions $a_u$ and $a_u^{(4)}$ that follow directly from
 the corresponding properties of the regularized determinants.  The first bounds can be stated as follows:
\begin{proposition} 
\label{stabilityestnew}
{\sl There exists  a positive  constant~$C$ such that the following estimates hold
\begin{equation} \label{eq1}  |a_u(\lam)| \leq e^{C \frac {|\lam|^2} {{ \rm Im}  (\lam^2) }\|u\|^2_{L^2(\R)}  }, \quad  |a^{(4)}_u(\lam)| \leq e^{C \frac {|\lam|^2} {{ \rm Im} (\lam^2) }\|u\|^2_{L^2(\R)}  } ,\quad \forall \lambda\in \Omega_+, \, u\in L^2(\R),\end{equation}  
\begin{equation}  \begin{aligned}
\label{eqstab} |a_{u_1} (\lam)-a_{u_2} (\lam)| & \leq C  e^{ C  \frac {|\lam|^2 } {{ \rm Im} (\lam^2)} \big(\|u_1\|^2_{L^2(\R)}+\|u_2\|^2_{L^2(\R)}\big)} \frac { |\lam| } {\sqrt { { \rm Im} (\lam^2) }  }  \|u_1- u_2\|_{L^2(\R)}, \\ & \qquad \qquad \qquad \qquad \qquad \qquad \qquad \quad \forall \lambda\in \Omega_+, \, u_1, \, u_2\in L^2(\R), \end{aligned}\end{equation} and
  \begin{equation} \begin{aligned}
\label{studyau}  &  \qquad  \bigl|a_{u} (\lam)-1\bigr| +  \bigl|a_{u}^{(4)} (\lam)-1\bigr| \leq Ce^{C \frac {|\lam|^4} {({ \rm Im}  (\lam^2))^2 }\|u\|^4_{L^2(\R)}  } \\
 & \times \frac { |\lam|^2} {\sqrt { { \rm Im} (\lam^2)} }\int_{\R} d p \, |\hat u (p)|^2  \biggl(\frac {1 }   { |p+ 2 \lam^2|^{ \frac 1 2}} + \frac {1 }   { |p- 2 \lam^2|^{ \frac 1 2}}\biggr), \,   \forall \lambda\in \Omega_+, \,  u\in L^2(\R). \end{aligned}
\end{equation}}
 \end{proposition}
     \begin{proof}    The two first estimates \eqref{eq1} and \eqref{eqstab} readily follow   from the relations\refeq{boundedest},  \refeq{detyest4} and\refeq{HSuseful},  \refeq{link24}.
  In order    to establish  the last estimate,   we start by observing that by virtue of\refeq{HSuseful},  \refeq{trace2useful} and \refeq{link24}, we have 
  \begin{equation} \begin{aligned}\label{*0}
    |a_u(\lam)-1| & \leq  e^{C \frac {|\lam|^2} {{ \rm Im}  (\lam^2) }\|u\|^2_{L^2(\R)}  }(|a_{u}^{(4)} (\lam)-1| +|\tr T^2_u(\lam)|)  \\ & \lesssim  e^{C \frac {|\lam|^2} {{ \rm Im}  (\lam^2) }\|u\|^2_{L^2(\R)}  }\Big(|a_{u}^{(4)} (\lam)-1| +\frac { |\lam|^2} {\sqrt { { \rm Im} (\lam^2)} } \int_{\R} d p \,   \frac {|\hat u (p)|^2  }   { |p+ 2 \lam^2|^{ \frac 1 2}} \Big) \cdotp  \end{aligned} \end{equation} 
  It remains to  control $a_u^{(4)}-1$. To this end, we apply\refeq{detyest3}, which, in view of   \eqref{HSuseful} and\refeq{eq:controlnorml0},  leads to the following inequality 
\begin{equation}\label{*}
\bigl|a_{u}^{(4)} (\lam)-1\bigr|
\leq Ce^{C \frac {|\lam|^4} {({ \rm Im}  (\lam^2))^2 }\|u\|^4_{L^2(\R)}  }
\|T^2_u(\lam)\|^2_2.
\end{equation}
According to\refeq{formulaTu}, the operator $T^2_u(\lambda)$ has the form
$$ T^2_u(\lam)=  \lam^2  \left(
\begin{array}{ccccccccc}
(D+ {\lam }^2)^{-1} u  (D -\lam^2)^{-1} \overline {u}&0 \\  \\
0  &(D - {\lam }^2)^{-1} \bar u (D+  {\lam }^2)^{-1} u\end{array}
\right)  .$$
 Then, using the explicit kernel of $(D\pm\lam^2)^{-1}$ (see \eqref{kernel0}), we get by straightforward computations 
 $$  \|(D+ {\lam }^2)^{-1} u  (D -\lam^2)^{-1} \overline {u}\|_2^2\leq \frac{1}{\im \lam^2}\|u\|_{L^2(\R)}^2\|(D+2\lam^2)^{-1}u\|_{L^\infty(\R)}^2\, \cdotp$$
Since 
$$
\begin{aligned}
  \|(D+2\lam^2)^{-1}u\|^2_{L^\infty(\R)} & \leq \|(p+2\lam^2)^{-1/4}\hat u\|^2_{L^2(\R)} \int_\R \frac{dp}{|p+2\lam^2|^{\frac 3 2}} \\ & \lesssim\frac{1}{(\im \lam^2)^{{\frac 1 2}}}\|(p+2\lam^2)^{-1/4}\hat u\|_{L^2(\R)}^2\virgp\end{aligned}$$ 
we infer that
\begin{equation}
\label{compl1}
\|(D+ {\lam }^2)^{-1} u  (D -\lam^2)^{-1} \overline {u}\|_2^2 \lesssim\frac{1}{(\im \lam^2)^{\frac 3 2}}\|u\|_{L^2(\R)}^2\|(p+2\lam^2)^{-1/4}\hat u\|_{L^2(\R)}^2\,  \cdotp
\end{equation}
Similarly, we have 
\begin{equation}
\label{compl2}
 \|(D- {\lam }^2)^{-1} \bar u  (D +\lam^2)^{-1}  {u}\|_2^2\lesssim \frac{1}{(\im \lam^2)^{\frac 3 2}}\|u\|_{L^2(\R)}^2\|(p-2\lam^2)^{-1/4}\hat u\|_{L^2(\R)}^2\,  \cdotp
 \end{equation}
 Combining the two latter inequalities together with \eqref{*0}-\eqref{*},  we get the desired bound \eqref{studyau}.
 \end{proof}
 
 \medbreak

 Invoking   the asymptotic formula\refeq{lim} together with the  stability estimate \eqref{eqstab}, we obtain the following corollary:
  \begin{cor} 
    \label{coruse}
    {\sl Let $u$ be a function in $L^{2}(\R)$. Then, for any $0<\delta<\frac\pi 2$,  there holds
  \begin{equation}
\label{studyaubeh0}    
\lim\limits_{\lam\rightarrow 0, \lam \in \Gamma_\delta}a_{u} (\lam)= 1 \, , \end{equation} 
and
\begin{equation}
\label{lim1}   \lim\limits_{|\lam|\rightarrow \infty, \lam \in \Gamma_\delta} a_{u} (\lam)=e^{-\frac{i}2\|u\|^2_{L^2(\R)}} ,
\end{equation} 
where we denote $\Gamma_\delta\eqdefa\big\{ \lam \in \Omega_{+}  :  \delta <  {\rm arg} (\lam^2)  < \pi -\delta \big\}\cdot$
  }
    \end{cor}
 \medbreak    
    
 \begin{remark} 
{\sl Observe also that the stability estimate \eqref{eqstab} combined with the $H^{\frac12}$ continuity of the DNLS flow gives 
 the conservation of $a_{u(t)}(\lambda)$ for $H^{\frac12}(\R)$-solutions of DNLS.  }
 \end{remark}
  \medbreak
 
 Assuming that the potential $u$ is in $H^{\frac 1 2}(\R)$, or more generally in   $L^2(\R)\cap L^4(\R)$, one gets:
   \begin{lemma} 
\label{lema4first}
{\sl There exists a positive  constant~$C$ such that:
\begin{equation} \label{eq5}     |a^{(4)}_u(\lam)- 1| \leq C e^{C \frac {|\lam|^4} {({ \rm Im}  (\lam^2) )^2}\|u\|^4_{L^2(\R)}  }   \frac {|\lam|^4} {({ \rm Im}  (\lam^2) )^3}\|u\|^4_{L^4(\R)} , \quad \forall \lam\in \Omega_+, \, \forall u\in L^2(\R)\cap L^4(\R),  \end{equation}  
\begin{equation}\label{eq6}
 |a_u(\lam)e^{\frac{i}2\|u\|^2_{L^2(\R)}}- 1| \leq Ce^{C\frac { |\lam|^4 }
  {({ \rm Im}  (\lam^2) )^2}  \|u\|^4_{L^2(\R)}}    \frac {|\lam|^2} {({ \rm Im}  (\lam^2) )^2}
  \|u\|^2_{\dot H^{ \frac 1 2}(\R)}  , \quad \forall \lam\in \Omega_+, \, \forall u\in H^{\frac12}(\R), \end{equation} 
and   \begin{equation}
\label{eq7} |a^{(4)}_{u_1} (\lam)-a^4_{u_2} (\lam)| \leq  C e^{C  \frac {|\lam|^4 } {({ \rm Im} (\lam^2))^2}  \big(\|u_1\|^4_{L^2(\R)}+\|u_2\|^4_{L^2(\R)} \big) }   \frac { |\lam|^{ \frac 1 2  } } { ({ \rm Im} (\lam^2))^{ \frac 3 8 } }  \|u_1- u_2\|^{ \frac 1 2  } _{L^4(\R)},\end{equation}
for all
$u_1, u_2$ in~$L^2(\R) \cap L^4(\R)$ and all $\lam$ in $\Omega_{+} $. 
}
\end{lemma}
\begin{proof} 
To prove the first inequality,  we use
\refeq{dety1} which according to  the fact that the matrix $U$ is anti-diagonal gives
$$ a^{(4)}_u(\lam)={\rm det}_6 ({\rm I}-T_u(\lam))  \exp \Big( - \frac {\tr T^4_u(\lam)} {4}\Big) \,  \cdotp$$
Invoking\refeq{detyest3},  we deduce  that there is a   positive constant  C such that $$ |a^{(4)}_u(\lam)- 1| \leq C e^{C \|T_u(\lam)\|^4_4} \big( \|T_u(\lam)\|^4 \|T_u(\lam)\|^2_2 + | \tr T^4_u(\lam) |   \big)  \,  \cdotp$$
Then, taking advantage of\refeq{HSuseful} and \eqref{eq:controlnorml0}, we infer that, for all $u$ in $L^2(\R) \cap L^4(\R)$, there holds
\begin{equation} \label{eq4}     |a^4_u(\lam)- 1| \leq C e^{C\frac {|\lam|^4} {({ \rm Im}  (\lam^2) )^2}\|u\|^4_{L^2(\R)}  }  \Big(\frac {|\lam|^4} {({ \rm Im}  (\lam^2) )^3}\|u\|^4_{L^4(\R)}  +     | \tr T^4_u(\lam)   |\Big) \cdotp \end{equation}  
We next compute $\tr T_u^4(\lam)$. In view of\refeq{formulaTu}, we have 
$$ \begin{aligned}  &T^4_u(\lam) =    \lam^4      \left(
\begin{array}{ccccccccc}
A(\lam)&0  \\  
0 & B(\lam)\end{array}
\right)  \, ,  \end{aligned}  $$
with 
\begin{eqnarray*}A(\lam)&=&(D + \lam^2)^{-1}u \, (D-  \lam^2)^{-1} \overline {u}(D + \lam^2)^{-1}u \, (D- \lam^2)^{-1} \overline {u} \\
B(\lam)&=&(D- \lam^2)^{-1} \overline {u}(D+ \lam^2)^{-1}u(D- \lam^2)^{-1} \overline {u}(D + \lam^2)^{-1}u \, .\end{eqnarray*} 
One can easily check  that  
$$ \begin{aligned}   \tr B(\lam)=\tr A(\lam)=  &  = \frac 1 {(2\pi)^2}  \int_{\R^4} d p dp_1dp_2 dp_3 \frac  { \hat u (p-p_1) \hat {\overline u} (p_1-p_2) \hat u (p_2-p_3) \hat {\overline u} (p_3-p) }
 {(p+\lam^2)(p_1-\lam^2)(p_2+\lam^2)(p_3-\lam^2)}  
\\ & =  \frac 1 {(2\pi)^2}  \int_{\R^3} d p dp_1dp_2   \hat u (p-p_1) \overline {\hat  u (p_2-p_1) }\hat u (p_2) \overline {\hat u (p)  }
 \, \cI(p, p_1, p_2, \lam^2)  \,   \virgp  \end{aligned} $$
where $$\cI(p, p_1, p_2, \lam^2) =  \int_{\R}  \frac  {d p_3 }
 {(p+p_3 +\lam^2)(p_1+p_3-\lam^2)(p_2+p_3 +\lam^2)(p_3-\lam^2)} \,  \cdotp$$
 Applying  the  Cauchy's residue theorem, we get
$$\cI(p, p_1, p_2, \lam^2) =    - 2i\pi \Big[    \frac  { 1}
 {(p_2+2\lam^2)(p- p_1+2 \lam^2)(p_2-p_1+2\lam^2)}  + \frac  { 1}
 {(p_2+2\lam^2)(p- p_1+2 \lam^2)(p+ 2 \lam^2)}\Big] 
   \cdotp  $$
This implies  that   
 $$
 \begin{aligned}
&  \tr T^4_u(\lam)  = - \frac {i\lam^4} {\pi }    \int_{\R^3} d p dp_1dp_2   \hat u (p-p_1) \overline {\hat  u (p_2-p_1)} \hat u (p_2) \overline {\hat u (p)  }
\\ &  \,  \quad \times  \Big[    \frac  { 1}
 {(p_2+2\lam^2)(p- p_1+2 \lam^2)(p_2-p_1+2\lam^2)}  + \frac  { 1}
 {(p+ 2 \lam^2)(p_2+2\lam^2)(p- p_1+2 \lam^2)}\Big]  \\
 & \qquad \qquad \qquad \qquad =-\frac{2i\lambda^4}{\pi}  \int_{\R^3} d p dp_1dp_2  \frac{ \hat u (p-p_1) \overline {\hat  u (p_2-p_1)} \hat u (p_2) \overline {\hat u (p)  }}
 {(p+ 2 \lam^2)(p_2+2\lam^2)(p- p_1+2 \lam^2)}\, \cdotp \end{aligned}
$$
Invoking Fourier-Plancherel formula, we deduce  that  \begin{equation}\label{tr4}
\tr T^4_u(\lam)=4i\lambda^4\int_\R dx\, {\bar u(x)}\left((D+2\lambda^2)^{-1}u(x)\right)^2(D-2\lam^2)^{-1}\bar u(x),\end{equation}
which, thanks to\refeq{kernel0}, shows that
$$
| \tr T^4_u(\lam)| \lesssim  \frac {|\lam|^4} {({ \rm Im}  (\lam^2) )^3}\|u\|^4_{L^4(\R)} .$$
According to\refeq{eq4},  this concludes the proof of \eqref{eq5}.

\medskip Let us now go to the proof of\refeq{eq6}.  For that purpose, we start by combining \eqref{HSuseful} together with\refeq{link24}, which implies that
\begin{equation}\label{l4.11}
\bigl| a_u(\lam)e^{\frac{i}2\|u\|^2_{L^2(\R)}} - 1\bigr| \leq   Ce^{C \frac {|\lam|^2} {{ \rm Im}  (\lam^2) }\|u\|^2_{L^2(\R)}  }\big(|a_{u}^{(4)} (\lam)-1| +|\tr T^2_u(\lam)-i\|u\|_{L^2(\R)}^2|\big).
\end{equation}
Since by
\eqref{trace2useful}, we have
\begin{equation}\label{l4.12}
\tr T^2_u(\lam)-i\|u\|_{L^2(\R)}^2=-i\int_{\R} d p \, \frac {p |\hat u (p)|^2 }   { p+ 2 \lam^2} \, \virgp \end{equation}  
we get \begin{equation}\label{l4.14}
\left|\tr T^2_u(\lam)-i\|u\|_{L^2(\R)}^2\right|\leq\frac{1}{2 \im (\lam^2)}\|u\|^2_{\dot H^{ \frac 1 2}(\R)} \, \cdotp \end{equation}  
Then invoking \eqref{eq5}, \eqref{l4.11}, \eqref{l4.14}
together with  the interpolation inequality
$$\|u\|^4_{L^{4}(\R)} \lesssim  \|u\|^2_{L^{2}(\R)} \|u\|^2_{\dot H^{ \frac 1 2}(\R)} \, ,$$
 we readily achieve the proof of the estimate \eqref{eq6}.
 
 \medskip
Finally, to establish \eqref{eq7}, we apply Estimate\refeq{detyest4} with $n=4$, which gives
\begin{equation*}\begin{split}
|a^{(4)}_{u_1} (\lam)-a^{(4)}_{u_2} (\lam)| &\leq C e^{ C(\|T_{u_1} (\lam)\|^4_4+ \|T_{u_2} (\lam)\|^4_4)}\|T_{(u_1-u_2)} (\lam)\|_4 \\
&\leq Ce^{ C(\|T_{u_1} (\lam)\|^4_4+ \|T_{u_2} (\lam)\|^4_4)}\|T_{(u_1-u_2)} (\lam)\|^{1/2}\|T_{(u_1-u_2)} (\lam)\|_2^{1/2}
 \, ,\end{split}\end{equation*}
which completes the proof of the estimate, thanks to\refeq{HSuseful} and\refeq{eq:controlnorml0}.
\end{proof} 

 \medbreak
 
 We will also need  the following  refinement  of the above estimates, that we formulate in terms of~$\ln\tilde a_u(\zeta)$ where, as above, $\tilde a_u(\zeta)=e^{\frac{i}2\|u\|^2_{L^2(\R)}}a_u(\sqrt{\zeta})$. 
 First note that by virtue of\refeq{lim1},~$\ln \tilde a_u(\zeta)$ is an holomorphic function of $\zeta$ in $\C_+$ with $\im \zeta$ sufficiently large
(depending on $\arg \zeta$), uniquely defined by the condition $\log \tilde a_u(\zeta)=o(1)$ as $|\zeta|\rightarrow \infty$.
In addition,  as soon as  $\|T_u(\sqrt{\zeta})\|<1$, it can be written as  a convergent series:
 \begin{equation} 
\label{loga}  \ln  \tilde a_u(\zeta) = \frac{i}2\|u\|^2_{L^2(\R)}- \sum^\infty_{k=2} \frac {\tr T_u^k(\sqrt{\zeta})} k \cdotp
\end{equation}
Denoting 
$$\Phi_{0, u}(\zeta)\eqdefa \frac{i}2\|u\|^2_{L^2(\R)}-\frac12 \tr T_u^2(\sqrt{\zeta})=\frac i 2
\int_{\R} d p \, \frac {p |\hat u (p)|^2 }   { p+ 2 \zeta}\, \virgp $$ we deduce that
\begin{equation}\label{gentr41}\Big|\ln  \tilde a_u(\zeta)-\Phi_{0, u}(\zeta)+\frac {\tr T_u^4(\sqrt{\zeta})} 4 \Big| \leq C \|T_u(\sqrt{\zeta})\|_2^2\|T_u(\sqrt{\zeta})\|^4,\end{equation}
provided that $\ds \|T_u(\sqrt{\zeta})\|\leq \frac12 \cdotp$

\medskip Note also that by \eqref{tr4},   for any $0\leq s\leq 1$,  there holds 
\begin{equation}\label{tr41}
\Big|\tr T_u^4(\sqrt{\zeta})+\frac{i}{2\zeta}\|u\|^4_{L^4(\R)}\Big|\leq C \frac{|\zeta|^2}{(\im \zeta)^{3+s}}\|u\|^3_{L^4(\R)}\|u\|_{\dot H^{\frac14+s}(\R)},
\,\, \forall \, u\in H^{\frac14+s}(\R),\,\,\zeta\in \C_+.
\end{equation}
Therefore, gathering the two latter estimates and 
taking into account  \refeq{HSuseful} and \eqref{eq:controlnorml0},  we obtain:

 \begin{lemma} 
\label{lema4first}
{\sl There exists  a positive constant $\kappa$ such that, for
any $\ds 0\leq s< \frac14\virgp$ one has: 
 $$\Big|\ln  \tilde a_u(\zeta)-\Phi_{0, u}(\zeta)- \frac{i}{8\zeta}\|u\|^4_{L^4(\R)} \Big|\leq C_s \left(1+\frac{|\zeta|}{\im \zeta}\|u\|^2_{L^2(\R)}\right)
\frac{|\zeta|^2}{(\im \zeta)^{3+s}}
\|u\|^3_{L^4(\R)}\|u\|_{\dot H^{\frac14+s}(\R)},$$
for all
$u\in H^{\frac14 +s}(\R)$ and all $\zeta\in \C_+$ satisfying $\ds \frac{|\zeta|}{(\im \zeta)^{3/2}}\|u\|^2_{L^4(\R)}\leq \kappa\virgp$ with some positive constant~$C_s$.}
\end{lemma}

\medbreak 

\subsubsection{Resolvent estimates} Consider the resolvent $L_u^{-1}(\lam)$.  For any $u$ in $L^2(\R)$, $L_u^{-1}(\lam)$
is a meromorphic function of $\lambda$ in $\Omega_+$ with values in the space of bounded operators on~$L^2(\R, \C^2)$, whose poles coincide with the zeros of $a_u$. In addition, it admits the following estimate:\begin{proposition} 
\label{resest}
{\sl There exists a positive  constant~$C$ such that, 
for all $u$ in $L^2(\R)$ and all $\lam$ in~$\Omega_{+}$, we have 
\begin{equation}
\label{inverseop}  \|L_u^{-1}(\lam)\|\leq \exp\Big( C \frac {|\lam|^2 } {{ \rm Im} (\lam^2)} \|u\|^2_{L^2(\R)} \Big) \frac {C } {|a_u(\lam)|\,  { \rm Im} (\lam^2)} \, ,\end{equation} provided that $a_u(\lam)\neq 0$.}
\end{proposition}
\begin{proof}
 Taking into account that
$${L}^{-1}_u(\lam)= ({\rm I}-T_u(\lam))^{-1} (\cL_0 - \lam^2)^{-1} \,, $$ 
the result follows immediately from Proposition \ref{defdetyrk} (3) and Identity \eqref{HSuseful}.  
\end{proof}

\medbreak 

\subsubsection{Bounds on the number of the eigenvalues of $L_u(\lam)$}
We start by observing that due to\refeq{size} and\refeq{mass1}, for any $u\in \mathcal{S}_{reg}$ and any $\theta \in ]0, \pi[$,
we have:
\begin{equation}
\label{controbound} \sharp \big\{ \zeta\in \C_+: \, \tilde  a_{u}(\zeta)= 0,  \, \,  \theta < \arg \zeta< \pi \big\}  \leq \frac{\|u\|^2_{L^2(\R)}}{4\theta} \cdotp \end{equation}
Furthermore, 
combining the density of $\mathcal{S}_{reg}$ in $L^2$ together  with the stability estimate 
\eqref{eqstab} and Corollary\refer{coruse}, 
we infer  that this inequality remains valid for $u\in L^2$. From\refeq{mass1} we also deduce:

 \begin{lemma} 
\label{intrel}
{\sl Let~$u$ be a function of   $\cS_{\rm reg} (\R)$  and $\theta \in ]0, \pi[$ such that $\tilde a_u(\zeta) \neq 0$, for all $\zeta$ in $\C_+$ with
$\arg \zeta = \theta$. Then  
\begin{equation}
\label{controlpert} 
\sharp \big\{ \zeta\in \C_+: \, \tilde  a_{u}(\zeta)= 0,  \, \,  \theta < \arg \zeta< \pi \big\}=
\frac{1}{2i\pi }\int\limits^{+ \infty \, e^{i\theta}}_0 \frac {\tilde a'_{u}(s)}  {\tilde  a_{u}(s)} ds+\frac1{4\pi}{\|u\|^ 2_{L^2(\R)}} \virgp\end{equation}
 where, all along this paper,    $ \ds \int\limits^{+ \infty \, e^{i\theta}}_0 ds $ denotes the integral along the path 
$\gamma \eqdefa \big\{z= \rho \, e^{i \theta},  \, \,  \rho \in \R_+ \big\}\cdot$}
\end{lemma}
\begin{proof} This lemma is a straightforward consequence of the analyticity of $\tilde a_u$  and of the  asymptotics\refeq{as}. Indeed, 
given  $u$ in $\cS_{\rm reg} (\R)$,
denote by $\zeta_j$, $j=1, \dots , N$,  the zeros of $\tilde a_u(\zeta)$ in the upper half plane and by $n$ the number of the zeros 
in the angle $\{\theta<\arg \zeta<\pi\}\cdotp$ Then, taking advantage of Formula\refeq{devol}, we easily get  
\begin{equation*}
\begin{split}
\int^{+ \infty e^{i\theta}}_0 ds  \frac {\tilde a'_{u}(s)}  {\tilde  a_{u}(s)}=& \sum\limits^N_{j=1} \int^{+ \infty e^{i\theta}}_0 ds \Big(\frac 1 {s -\zeta_j}  - \frac 1 {s - \overline \zeta_j}\Big) -
\frac 1 {2i\pi} \int^\infty_{-\infty} \frac {d\xi} {\xi} \, \log |\tilde a_u(\xi)|^2 \\
=&2i\pi n-2i\sum\limits_{j=1}^N\arg\zeta_j-
\frac 1 {2i\pi} \int^\infty_{-\infty} \frac {d\xi} {\xi} \, \log |\tilde a_u(\xi)|^2\\
=&2i\pi n-\frac{i}2{\|u\|^ 2_{L^2(\R)}}\, \virgp
\end{split}
\end{equation*}
which completes the proof of \eqref{controlpert}. 
\end{proof}
\medbreak

In the proof of Theorem \ref{Mainth},  we will also need   the following $L^2$  version  of Lemma\refer{intrel}.
\begin{lemma} 
\label{intrelrootsnew}
{\sl Let~$u$ be a function in $L^2$ and $\theta\in ]0,\pi[$ such that   $\tilde a_{u}(\zeta) \neq 0$, for all $\zeta$ belonging to the ray $e^{i\theta}\R_+$. Then,
\begin{equation}
\label{calcullimitnew} \frac{1}{2i\pi} \int\limits^{+ \infty \, e^{i \theta}}_{0}ds \, \frac { \tilde  a'_{u}(s)}  {\tilde  a_{u}(s)}+   \frac 1 {4 \pi} \|u\|^ 2_{L^2(\R)} \in \N \,\cdot\end{equation}}
\end{lemma}
\begin{proof}  
First observe that in view of  Corollary \ref{coruse}, the integral $\ds \int\limits^{+ \infty \, e^{i \theta}}_{0}ds \, \frac { \tilde  a'_{u}(s)}  {\tilde  a_{u}(s)}$ makes sense.  In order to establish \eqref{calcullimitnew},  we proceed by approximation: let $(u_n)_{n \in \N}$ be a sequence of  $\cS_{\rm reg} (\R)$ converging  to~$u$ in $L^2$. The stability estimate \eqref{eqstab} together with Corollary \ref{coruse}
ensures that, for~$n$ sufficiently large,   $\tilde a_{u_n}$ does not vanish on the ray $e^{i\theta}\R_+\virgp$  and one has
$$ \int\limits^{+ \infty \, e^{i \theta}}_{0}ds \, \frac { \tilde  a'_{u_n}(s)}  {\tilde  a_{u_n}(s)}
 \stackrel{n \to \infty}\longrightarrow  \int\limits^{+ \infty \, e^{i \theta}}_{0}ds \, \frac { \tilde  a'_{u}(s)}  {\tilde  a_{u}(s)}\cdotp$$
Since by virtue of \eqref{controlpert},
$$
\frac{1}{2i\pi} \int\limits^{+ \infty \, e^{i \theta}}_{0}ds \, \frac { \tilde  a'_{u_n}(s)}  {\tilde  a_{u_n}(s)}+   \frac 1 {4 \pi} \|u\|^ 2_{L^2(\R)} \in \N \virgp $$
Lemma\refer{intrelrootsnew} follows by passing to the limit~$n\rightarrow \infty$.
\end{proof}

\medbreak 

We next show that the real parts of the zeros of $\tilde a_u(\zeta)$ are low-bounded uniformly with respect to $u$ in bounded sets of
$H^{\frac12}$.
\begin{lemma}
\label{lemmcont}
{\sl  Let $u\in H^{\frac12}(\R)$. There exists a positive constant~$C$ depending only on~$\|u\|_{H^{\frac 1 2}(\R)}$, such that 
the function $\tilde a_u(\zeta)$ has no zeros in the region $\{\zeta\in \C_+: \, \re \zeta\leq -C \} \cdot$}
\end{lemma}
\begin{proof}
Let $\zeta_0\in \C_+$ be a zero of 
$\tilde a_u$. Then, there exists a function $\psi$ in  $H^{1}(\R)$, 
with~$\ds \|\psi\|_{L^2(\R)} =1$,  such that
\begin{equation}
\label{systlam1}i  \sigma_3 \partial_x \psi =  \lam^2_0    \psi+ i \lam_0 \left(
\begin{array}{ccccccccc}
0 &u  \\
  \overline {u}&0
\end{array}
\right)  \psi\,, \end{equation}
where $\lam_0=\sqrt{\zeta_0}\in \C_{++}=\{\lam\in \C:\,\, \re \lam>0, \, \im \lam>0\color{blue}\}$,
which  shows that, for
all $s>0$,
$$ \|\partial_x \psi\|_{L^{2}(\R)} \leq |\lam_0|^2+   C_s|\lam_0| \|u\|_{H^{\frac 1 2}(\R)}  \|\psi\|_{H^s(\R)}\, .$$
From this inequality we readily deduce   that
\begin{equation}
\label{estH1}  \|\psi\|_{H^{1}(\R)} \lesssim_{\|u\|_{H^{\frac12}(\R)}}(1+|\zeta_0|) \, .\end{equation}
Furthermore, taking the imaginary part of the scalar product  of the  identity \eqref{systlam1} with $\psi$, we get $$2 {\rm Re} (\lam_0)\, {\rm Im} (\lam_0) = -  \re \lam_0 \left\langle \left(
\begin{array}{ccccccccc}
0 &u  \\
  \overline {u} &0
\end{array}
\right)  \psi , \psi \right\rangle  \virgp$$
which  ensures    that 
\begin{equation}
\label{firstst}  \im \lam_0  \lesssim_s   \|u\|_{H^{\frac 1 2}(\R)}  \|\psi\|^{s}_{H^1(\R)} \lesssim_{s,  \|u\|_{H^{\frac 1 2}(\R)}}(1+|\zeta_0|)^{s}\, ,\end{equation}
for any $s>0$.
Consequently,
 \begin{equation}
\label{firststbis*} \re \zeta_0= \big(\re \lam_0\big)^2 - \big(\im \lam_0\big)^2\geq - C_{s,\|u\|_{H^{\frac 1 2}(\R)}}
(1+|\zeta_0|)^{2s} \, .\end{equation}
Since by \eqref{eq6}, $\zeta_0$ satisfies: $|\zeta_0|\lesssim_ {\|u\|_{H^{\frac 1 2}(\R)}} (1+|\re \zeta_0|)$,
the inequality \eqref{firststbis*}
implies that
$$\re \zeta_0\geq -C_{\|u\|_{H^{\frac 1 2}(\R)}} .$$
This completes the proof of the lemma.  \end{proof}
\medbreak

\subsubsection{Some additional  results}
Here we derive some consequences of Lemma \ref{lema4first} that will play an important role in the proof of Theorem \ref{Mainth}.
The first one relates the $\dot H^{\frac12}$ norm of the potential~$u$ to the $L^1$ norm of the function $\im \ln   \wt a_u$
on the imaginary half-axis  $i\R_+$.  Denoting 
 \begin{equation}
\label{defimpfunct} \varphi_u (\rho)=\im (\ln   \wt a_u(i\rho)), \end{equation}
 we have
\begin{proposition}
\label{crucialestleastnew}
{\sl Let  $u\in H^{\frac12}(\R)$. Then  the function  $\varphi_u$ belongs to
$\ds L^1([ {\|u\|_{L^4(\R)}^4}/{\kappa^2},+\infty [)$, where~$\kappa$ is the constant introduced in Lemma \ref{lema4first},  and  for any $\ds R\geq \frac 1 {\kappa^2}$ the following estimate holds: \begin{equation}  \label{lowerboundimpnew} \|u\|^{ 2}_{\dot H^{\frac 1 2}(\R)} \lesssim_{R, \|u\|_{L^2}} 
\|\varphi_u\|_{L^{1}([R \|u\|^4_{L^4},+\infty [)} +\|u\|^4_{L^4(\R)}.
\end{equation}}
\end{proposition}
\begin{proof}
By scaling argument, it is enough to prove the proposition assuming that $ \|u\|^4_{L^4(\R)}=1$. Denoting  $\ds \varphi_{0,u}(\rho)=\frac12\int_\R\frac{p^2|\hat u(p)|^2}{p^2+4\rho^2}dp$,  we deduce
 from  Lemma \ref{lema4first} that\begin{equation}\label{0*}
\big|\varphi_u(\rho)-\varphi_{0,u}(\rho)\big|\leq C_s\rho^{-1-s} \left(1+\|u\|^2_{L^2(\R)}\right)\|u\|_{\dot H^{\frac14+s}(\R)},
\end{equation}
provided that $\ds \rho \geq \frac 1 {\kappa^2}$. 
Observing that 
$$\|\varphi_{0,u}\|_{L^{1}([R, +\infty[)}= \frac 1 4 \int_\R|p||\hat u(p)|^2 \Big(\frac \pi 2 - \arctan \frac {2 R} {|p|}\Big) dp,
$$
we deduce that  for any $R>0$, 
\begin{equation}\label{estint}  
\|\varphi_{0,u}\|_{L^{1}(\R_+)}\lesssim  \|u\|^2_{\dot H^{\frac12}(\R)}\lesssim_R \|u\|^2_{L^2(\R)}+ \|\varphi_{0,u}\|_{L^{1}([R, +\infty[)}
\, .\end{equation}

Therefore, invoking the estimate\refeq{0*}  and integrating with respect to  $\rho$, we infer that the function~$\varphi_u$ belongs to $
L^1([ 1 /{\kappa^2},+\infty [)$, and  that we have, for any $\ds R \geq \frac 1 {\kappa^2}\virgp$   
$$ \|u\|^{2}_{\dot H^{\frac 1 2}(\R)} \lesssim_{R, \|u\|_{L^2}(\R)} 1+
\|\varphi_u\|_{L^{1}([R,+\infty [)},$$
which  completes the proof of the proposition. 
\end{proof}

\medbreak

We conclude this subsection by the following rigidity result concerning the zero-free case.
\begin{lemma} 
\label{studyfirstpartimaginary}
{\sl Let $u\in H^{\frac 1 2}(\R)$ be such that the corresponding function $\tilde a_u$ has no zeros in 
 ~$\C_+$. Then
 \begin{equation} \label{impartcoef}   \varphi_u(\rho) \geq 0 , \quad \, \forall\ \rho\geq 0.\end{equation}  
 If in addition, 
 $\varphi_u(\rho_0) =0$ for some $\rho_0>0$,  then
 \begin{equation} \label{impartcoef1}  \tilde a_u(\zeta) = 1\, , \quad  \forall \zeta \,\in \C_+ \,. \end{equation}  }
\end{lemma}
\begin{proof} In order to establish the lemma, we proceed by approximation: let $(u_n)_{n \in \N}$ be a sequence in 
$\cS_{\rm reg}$ that converges to $u$ in $H^{\frac 1 2}(\R)$. We denote by $\zeta_j^n$, $j=1, \dots , N_n$,  the zeros of 
$\tilde a_{u_n}(\zeta)$ in $\C_+$.
Combining  Corollary \ref{coruse}  with the fact that $\tilde a_u$ does not vanish on $\C_+$, and taking into account
the stability estimate\refeq{eqstab}, we infer  that
 \begin{equation} \label{estzeros}  \sup\limits_{j=1, \dots, N_n}\frac {\im \zeta^n_j}  {|\re \zeta^n_j|} \stackrel{n\to\infty}\longrightarrow 0 \, \cdotp\end{equation}   
Therefore, invoking  \eqref{devol}, we deduce that,  for $n$ sufficiently large, 
 \begin{equation} \label{formparticular} 
 \begin{aligned}  &   \varphi_{u_n}(\rho)= \sumetage {1 \leq j \leq N_n} { \re \zeta^n_j < 0}  \im \ln \Big( \frac {i\rho -\zeta^n_j}  {i\rho - \overline \zeta^n_j} \Big) +\sumetage {1 \leq j \leq N_n} { \re \zeta^n_j > 0}  \im  \ln \Big( \frac {i\rho -\zeta^n_j}  {i\rho - \overline \zeta^n_j}\Big) \\ &\qquad  \qquad \qquad \qquad \qquad  \qquad  \qquad \qquad -\frac 1 {2\pi} \int^\infty_{-\infty} \frac {d\xi} {{\xi} ^2 +\rho^2} \, \xi\,  \ln |\tilde a_{u_n}(\xi)|^2, \quad \forall \, \rho>0 \, \cdotp\end{aligned} 
 \end{equation}
In view of\refeq{size},   we have
$$-\frac 1 {2\pi} \int^\infty_{-\infty} \frac {d\xi} {{\xi} ^2 +\rho^2} \, \xi \,  \ln |\tilde a_{u^{(n)}}(\xi)|^2 \geq 0 \, \cdotp$$
Note also  that, by virtue of \eqref{estzeros}, $\ds \im \ln \Big( \frac {i\rho -\zeta^n_j}  {i\rho - \overline \zeta^n_j}\Big) $ has the same   sign as 
$\re \zeta_j^{(n)}$ and
 \begin{equation*}  \sup\limits_{{j=1, \dots, N_n,\atop \rho>0}}\Big| \im \ln \Big( \frac {i\rho -\zeta^n_j}  {i\rho - \overline \zeta^n_j}\Big) \Big|
  \stackrel{n\to\infty}\longrightarrow 0 \,.\end{equation*} 
 Then, taking into account the bound
 \begin{equation} \label{cardzero}\sharp \big\{\zeta^{n}_i\,, { \rm Re} (\zeta^{n}_j) < 0 \big\} \lesssim  \|u^{(n)}\|^2_{L^2(\R)}   \lesssim \|u\|^2_{L^2(\R)}  \, ,\end{equation} 
 that readily follows   from\refeq{controbound},
 we conclude that for any $\rho>0$, 
 \begin{equation} \label{formparticular1} 
   \varphi_{u_n}(\rho)= \underbrace{\sumetage {1 \leq j \leq N_n} { \re \zeta^n_j > 0}  { \rm Im}  \ln \Big( \frac {i\rho -\zeta^n_j}  {i\rho - \overline \zeta^n_j} \Big)}_{\geq 0}\underbrace{-\frac 1 {2\pi} \int\limits^\infty_{-\infty} \frac {d\xi} {{\xi} ^2 +\rho^2} \, \xi\,  \ln |\tilde a_{u_n}(\xi)|^2}_{ \geq 0} + o(1),   \quad n\rightarrow +\infty \, \virgp 
   \end{equation}
 which gives \eqref{impartcoef} after
 passing to the limit $n\rightarrow +\infty$.

 \medskip Assume now that  $\varphi_u(\rho_0)=0$, for some $\rho_0>0$. 
 It follows then from\refeq{formparticular1} that\begin{equation}
\label{behaviorlimitebis} \int^\infty_{-\infty} \frac {d\xi} {{\xi} ^2 +  \rho_0^2} \, \xi\,  \ln |\tilde a_{u_n}(\xi)|^2 \stackrel{n \to  + \infty} \longrightarrow 0  \, \virgp\end{equation} 
 and
 \begin{equation}
 \label{behaviorlimitebis0}
 \sumetage {1 \leq j \leq N_n} { { \rm Re} (\zeta^n_j) > 0}  { \rm Im}  \ln \Big( \frac {i\rho_0-\zeta^n_j}  {i\rho_0 - \overline \zeta^n_j} \Big)  \stackrel{n \to  + \infty} \longrightarrow 0 \, \cdotp
 \end{equation}
 Recall that 
   $$\|u_n\|^2_{L^2(\R)}  = 4 \sum^{N_n}_{j=1} { \rm arg} (\zeta^n_j)   \underbrace{- \frac 1 { \pi} \int^\infty_{-\infty} \frac  {d\xi} {\xi} \, \ln |\tilde a_{u_n}(\xi)|^2 }_{ \geq 0} \,  \virgp$$
   which, thanks to \eqref{estzeros}, implies that, for all $n$ sufficiently large:
   \begin{equation}
\label{behaviorlimitebis1}   
   - \int^\infty_{-\infty} \frac  {d\xi} {\xi} \, \ln  |\tilde a_{u_n}(\xi)|^2\lesssim \|u\|^2_{L^2(\R)} \, \virgp
   \end{equation}
   and 
  \begin{equation}
  \label{behaviorlimitebis2}
   \sumetage {1 \leq j \leq N_n} { \re \zeta^n_j > 0}\frac{\im \zeta^n_j}{\re \zeta^n_j}\lesssim \|u\|^2_{L^2(\R)} \, \cdotp
 \end{equation}
Invoking  \eqref{behaviorlimitebis1} and applying Cauchy-Schwarz inequality, we readily gather that 
$$
\Big|\int^\infty_{-\infty} \frac {d\xi} {\xi-i\rho} \, \ln |\tilde a_{u_n}(\xi)|^2\Big|^2 \lesssim \|u\|^2_{L^2(\R)} \Big(\sup_{\xi \in \R}\frac {{\xi} ^2 +  \rho_0^2} {{\xi} ^2 +  \rho^2}\Big)\int^\infty_{-\infty} \frac {d\xi} {{\xi} ^2 +  \rho_0^2} \, \xi\,  \ln |\tilde a_{u_n}(\xi)|^2\, \virgp
$$
which according  to \eqref{behaviorlimitebis} ensures that 
\begin{equation}
\label{behaviorlimitebis3}
\int^\infty_{-\infty} \frac {d\xi} {\xi-i\rho} \, \ln |\tilde a_{u_n}(\xi)|^2 \stackrel{n \to  + \infty} \longrightarrow 0 \virgp \quad \forall \rho>0  \, \cdotp
\end{equation}
Applying again Cauchy-Schwarz inequality, we easily get,  for any $\rho>0$,
$$ \sumetage {1 \leq j \leq N_n} { \re \zeta^n_j > 0} \frac{\im \zeta^n_j}{|i\rho - \overline \zeta^n_j|}
  \lesssim_\rho\left( \sumetage {1 \leq j \leq N_n} { \re \zeta^n_j > 0}\frac{\im \zeta^n_j}{\re \zeta^n_j}\right)^{1/2}
  \left(\sumetage {1 \leq j \leq N_n} { \re \zeta^n_j> 0}  \frac{\im \zeta^n_j \re \zeta^n_j}{|i\rho_0 - \overline \zeta^n_j|^2} \right)^{1/2}\cdotp$$
Combining \eqref{estzeros}, \eqref{behaviorlimitebis0} together with  \eqref{behaviorlimitebis2}, we infer    that
 for all $\rho>0$,  
 \begin{equation}
 \label{1}
 \sumetage {1 \leq j \leq N_n} { \re \zeta^n_j > 0} \frac{\im \zeta^n_j}{|i\rho - \overline \zeta^n_j|}
 \lesssim_\rho\left( \sumetage {1 \leq j \leq N_n} { \re \zeta^n_j > 0}\frac{\im \zeta^n_j}{\re \zeta^n_j}\right)^{1/2}
  \left(\sumetage {1 \leq j \leq N_n} { \re \zeta^n_j> 0}  {\im}  \ln \Big( \frac {i\rho_0-\zeta^n_j}  {i\rho_0 - \overline \zeta^n_j} \Big) \right)^{1/2}
  \stackrel{n \to  + \infty} \longrightarrow 0.
  \end{equation}
  Finally, in view of \eqref{estzeros} and \eqref{cardzero}, we have 
  \begin{equation}
 \label{2} 
   \sumetage {1 \leq j \leq N_n} { \re \zeta^n_j <0} \frac{\im \zeta^n_j}{|i\rho - \overline \zeta^n_j]}
   \leq   \sumetage {1 \leq j \leq N_n} { \re \zeta^n_j <0} \frac{\im \zeta^n_j}{|\re \zeta^n_j|}   \stackrel{n \to  + \infty} \longrightarrow 0 \,  \cdotp
 \end{equation}   
 As an immediate consequence of the latter estimates\refeq{behaviorlimitebis3}-\eqref{2}, we obtain according to\refeq{devol}   that 
 $\tilde a_{u_n}(i\rho) \stackrel{n \to  + \infty} \longrightarrow 1$, for any $\rho>0$.  Therefore,  $\tilde a_u\equiv1$  on $i\R_+$.
 The analyticity of $\tilde a_u$  ensures then that $\tilde a_u\equiv1$ on $\C_+$.
 \end{proof}

\medbreak
 
 \begin{remark} \label{remm}
 { \sl It follows from Corollary \ref{coruse}, that $\tilde a_u\equiv1$ on $\C_+$ implies $\|u\|_{L^2}^2\in 4\pi \N$. Let us also mention that
 the set of potentials $u\in H^{\frac12}(\R)$ verifying \eqref{impartcoef1} is not trivial: it contains the algebraic solitons $u_{0,c}$. We shall denote this set by ${\mathcal A}$.}
 \end{remark}

\medbreak

\subsection{B\"acklund transformation} \label{defBacklund transformation}  In this paragraph,  we introduce    the B\"acklund transformation for the Kaup-Newell
spectral problem \eqref{sp} in the form needed for the proof of Theorem \ref{Mainth}, following closely \cite{PSS} (see also
\ccite{SHW}  and the references therein).
Given $u\in \mathcal{S}(\R)$, $\lambda\in \C_{++}$  and~$\eta=\left(
\begin{array}{ccccccccc}
\eta_1  \\
\eta_2
\end{array}
\right) $   a non zero smooth solution of  the Kaup-Newell spectral problem $L_u(\lambda)\eta=0$, 
one defines 
the
 B\"acklund transformation $\cB_{\lam} (\eta)$  by\begin{equation}
\label{backtransf} \cB_{\lam} (\eta) u \eqdefa G_{\lam} (\eta) \Big[G_{\lam} (\eta) \, u- \cS_{\lam} (\eta)  \Big] \, ,\end{equation}
where 
\begin{equation}  \label{defS} G_{\lam} (\eta) = \frac {d_{\overline \lam} (\eta)}  {d_\lam (\eta)}, \quad \cS_{\lam} (\eta) = 2i (\lam^2- \overline \lam^2) \frac {\eta_1 \overline \eta_2} {d_\lam (\eta)}\,  \virgp \end{equation}
with $d_\lam (\eta)= \lam |\eta_1|^2+ \overline \lam |\eta_2|^2 $. Since $\eta$ depends  implicitly of $u$,  the transformation \eqref{backtransf}  is nonlinear with respect to the function $u$. One can easily check  that
$\cB_{\lam} (\eta)u\in \mathcal{S}(\R)$. Observe also that
 \begin{equation}\label{SLinfty}
|G_\lam(\eta)| = 1\quad {\rm and}\quad
 |\cS_{\lam} (\eta)|  \leq 4 \im \lam\,  \cdot\end{equation}
Moreover, by straightforward computations, one can check that (see Appendix\refer{proofderG} for the proof)
  \begin{equation}\label{derG}\Big|\frac d {dx}G_{\lam} (\eta) (x)\Big| \leq 8 (\im\lam)^2 + 4\im(\lam) |u(x)| \, \cdotp \end{equation}

The key property of the B\"acklund transformation \eqref{backtransf}  is that it allows to add or to remove eigenvalues of the 
Kaup-Newell spectral problem  without   changing   the scattering coefficient~$b_u(\lambda)$. In particular, assume that $a_u(\lambda)$ has a simple zero $\lambda_1\in \C_{++}$ and let 
$\eta \in L^2(\R, \C^2)\setminus\{0\}$ be the corresponding eigenfunction: 
 $L_u(\lambda_1)\eta=0$.
 Then the scattering coefficients associated to the potential
$u^{(1)} \eqdefa \cB_{\lam_1} (\eta) u$ are given by (see for instance \ccite{SHW, PSS})\footnote{According to\refeq{mass1}, we thus have  $\tilde a_{u^{(1)}}(\zeta)= \tilde a_{u} (\zeta)  \frac {\zeta  -  \overline \zeta_1}  {\zeta -  \zeta_1}\cdotp $}
\begin{equation}
\label{backtransfrela} a_{u^{(1)}}(\lam)= a_{u} (\lam) \frac {\lam^2_1}  {\overline \lam^2_1}\,  \frac {\lam^2 -  \overline \lam^2_1}  {\lam^2 -  \lam^2_1},\, \quad b_{u^{(1)}}(\lam)= b_u(\lambda).\end{equation}
Thus, $a_{u^{(1)}}$ does not vanish at $\pm \lambda_1$.

\medbreak 

\section{Proof of the main theorem}\label {proofmainth} 
\subsection{Strategy of proof}
We start by  a brief overview of the main ideas involved in the proof of Theorem \ref{Mainth}.
By the local well-posedness result of Takaoka, Theorem \ref{Mainth} amounts to showing that any $H^{\frac12}$
solution $u$ of the DNLS equation,  defined on a time interval $I$, satisfies
\begin{equation}\label{b1}
\sup\limits_{t\in I}\|u(t)\|_{H^{\frac 1 2}(\R)}<+\infty.\end{equation}
To establish Property \eqref{b1},  we combine the integrability structure of DNLS with the profile decomposition techniques, proceeding by contradiction. Namely, assuming  that there exists~$u_0$ in~$H^{\frac 1 2}(\R)$ generating a solution $u\in C([0, T[, H^{\frac12}(\R))$  of the DNLS equation that verifies
$$\sup\limits_{0\leq t< T}\|u(t)\|_{H^{\frac 1 2}(\R)}=+\infty,
$$ we take a sequence $(t_n)_{n \in \N} \subset [0, T[$ such that~$\|u(t_n)\|_{H^{\frac 1 2}(\R)}\to +\infty$, as $n$ goes to infinity. 
Then, setting  $$\ds U_n(x)= \frac 1 {\sqrt{\mu_n}} u(t_n, \frac x { \mu_n}) \with \mu_n=\|u(t_n)\|^2_{\dot H^{\frac 1 2}(\R)},$$
we start by analyzing the profile decomposition of the sequence $(U_n)$ with respect to
the Sobolev embedding $H^{\frac 1 2}(\R) \hookrightarrow L^{p}(\R)$, for~$2<p< \infty$. Using  the conservation of 
$a_u$,  we show  that this decomposition contains at least one non-zero profile, that the number of profiles is bounded by~$\ds \frac{\|u_0\|^2_{L^2(\R)}} {4 \pi}$,
and that all of them
 belong to 
the set ${\mathcal A}$. 
This rigidity property is established in Section \ref{preliminarystatement1}  and relies heavily on the results of Section \ref{preliminarystatementscat}. 
With such a decomposition at hand, we then show, making use of the  B\"acklund transformation and of Lemmas 
\ref{intrel},
\ref{intrelrootsnew},
that  up to a subsequence and a suitable regularization, the function  $a_{U_n}$ admits a zero~$z_{n}$ satisfying~${\rm Re} (z^2_{n})= c_0 < 0$.  To conclude the proof, it remains to
invoke the scaling property\refeq{sca} that implies that~$a_{u}(z_n  \sqrt{\mu_n})= 0 $, which is in contradiction with
our assumption $\|u(t_n)\|_{H^{\frac 1 2}(\R)}\stackrel{n\to\infty}\longrightarrow +\infty$,  since,    by Lemma\refer{lemmcont}, any zero $z$ of~$a_{u}$ satisfies $ { \rm Re} (z^2) \geq -   C$, for some positive constant $C$ depending only on~$\|u_0\|_{H^{\frac 1 2}(\R)}$. 

\subsection{Rigidity type results}\label {preliminarystatement1} 
\subsubsection{Profile decompositions}
The first step in the proof of Theorem\refer{Mainth} consists in establishing the following profile decomposition for solutions violating the bound \eqref{b1}, assuming that such solutions exist.
\begin{theorem}
\label{result1newth}
{\sl Assume that for some $u_0\in H^{\frac 1 2}(\R)\setminus\{0\}$,  the solution $u\in C([0, T^*[, H^{\frac12}(\R))$ of the~DNLS equation with 
initial data $u(0)=u_0$ verifies $\sup\limits_{0\leq t< T^*}\|u(t)\|_{H^{\frac 1 2}(\R)}=+\infty$. Let $(t_n)_{n \in \N} $   be such that $\|u(t_n)\|_{H^{\frac 1 2}(\R)}\stackrel{n\to\infty}\longrightarrow +\infty$, and set $\ds U_n(x)= \frac 1 {\sqrt{\mu_n}} u(t_n, \frac x { \mu_n})$ with $\mu_n=\|u(t_n)\|^2_{\dot H^{\frac 1 2}(\R)}$. Then
there exist  an integer\footnote{In particular, $\|u_0\|^2_{L^2(\R)}\geq 4\pi$.} $\ds 1 \leq L_0 \leq   \frac{\|u_0\|^2_{L^2(\R)}} {4 \pi}\virgp$  a family of functions $(V^{(\ell)})_{1 \leq \ell \leq L_0}$ in ${\mathcal A}\setminus \{0\}$ and a family of orthogonal cores\footnote{Following the terminology of Patrick G\'erard in\ccite{pgerard1}, we designate by a core $\underline{y}^{(\ell)}$ any   real sequence $(y^{(\ell)}_n)_{n \in \N} $.}  $(\underline{y}^{(\ell)})_{\ell \geq 1}$,   
 in the sense that for all $\ell \neq \ell'$, we have $ |y^{(\ell)}_n-  y^{(\ell')}_n| \stackrel{n\to\infty}\longrightarrow \infty$, 
 such that,
up to a subsequence,  
$$ U_n(y)= \sum_{\ell=1}^{L_0}  V^{(\ell)}(y-y^{(\ell)}_n)  + {\rm
r}_n(y) ,$$
where
$$\lim_{n\to\infty}\;\|{\rm
r}_n\|_{L^{p}(\R)}= 0,$$
for all $2<p<\infty$.

\medskip Furthermore, 
\begin{equation} \label{ortogonalth} 
 \|
\chi(D) U_n\|_{L^2(\R)}^2=\sum_{\ell=1}^{L_0} \| \chi(D) V^{(\ell)}\|_{L^2(\R)}^2+\|\chi(D)
 {\rm
r}_n\|_{L^2(\R)}^2+\circ(1),\quad   n\to\infty \, ,\end{equation}
for any function $\chi\in <p>^{1/2}L^\infty(\R)$.}
\end{theorem}

\medbreak 

We begin the proof of Theorem \ref{result1newth} with the following proposition.

\begin{proposition}
\label{gendecompo}
{\sl With the previous notations, there exist a sequence of profiles $( {V^{(\ell)} })_{\ell \geq 1}$  in~$H^{\frac 1 2}(\R)$ which are not all zero  and may be regarded as ordered by decreasing $H^{\frac12}$ norm, 
 and a sequence of orthogonal cores $(\underline{y}^{(\ell)})_{\ell \geq 1}$
  such that, up to a subsequence extraction,  we have for all~$L\geq 1$,
\begin{equation} \label{decompNgen}U_n(y) = \sum_{\ell=1}^{L} V^{(\ell)}(y-y^{(\ell)}_n) +{\rm
r}_n^{L}(y),\end{equation}
where
\begin{equation} \label{decompNgen1}
  \limsup_{n\to\infty}\;\|{\rm
r}_n^{L}\|_{L^p(\R)}\stackrel{L\to\infty}\longrightarrow 0\, ,\end{equation} 
for all $2<p<\infty$.
In addition,   \begin{equation} \label{ortogonalth} \|
\chi(D) U_n\|_{L^2(\R)}^2=\sum_{\ell=1}^{L} \| \chi(D) V^{(\ell)}\|_{L^2(\R)}^2+\|\chi(D)
 {\rm
r}_n^L\|_{L^2(\R)}^2+\circ(1),\quad   n\to\infty \, ,\end{equation}
for any $\chi\in <p>^{1/2}L^\infty(\R)$ and any $L\geq 1$.
 }\end{proposition}
\begin{proof}
 Since the sequence  $(U_n)_{n \in \N}$ is  bounded  in~$H^{\frac 1 2}(\R)$,  
 following the work of P. Gerard \ccite{pgerard1},  (see Proposition~4, \ccite{pgerard1}), 
 we infer that there exist a sequence of profiles $( {V^{(\ell)} })_{\ell \geq 1}$  in~$H^{\frac 1 2}(\R)$
 and a sequence of orthogonal cores  $(\underline{y}^{(\ell)})_{\ell \geq 1}$   such that, up to a subsequence extraction,  the properties\refeq{decompNgen}-\eqref{ortogonalth} are satisfied. 
 To conclude the proof of the proposition, 
 it remains to show that  the   profile decomposition\refeq{decompNgen} includes at least one profile $V^{(\ell)}\neq 0$.
 To this end, we will use Proposition\refer{crucialestleastnew}. Since  $\mu_n\stackrel{n\to\infty}\longrightarrow +\infty$, 
 combining the estimate \eqref{lowerboundimpnew} with
 the conservation of $a_u$, we deduce that
 $$ \|u(t_n)\|^4_{L^4(\R)}\geq c\mu_n\quad \forall n\in\N, $$
 for some  positive constant $c$ depending on the initial  data.
 Therefore, $ \|U_n\|^4_{L^4(\R)}\geq c$, which ensures the existence of at least one non zero profile.
      \end{proof} 
      
      \medbreak
      
      \begin{remark}\label{rem2}
 {\sl    Note that thanks to \eqref{ortogonalth}, we have
    $$\sum_{\ell=1}^{\infty} \|  V^{(\ell)}\|_{L^2(\R)}^2\leq \|u_0\|^2_{L^2(\R)},$$
    and 
   $$\sum_{\ell=1}^{\infty} \|  V^{(\ell)}\|_{\dot H^{\frac12}(\R)}^2\leq 1.$$ }
   \end{remark}
\medbreak

\subsubsection{Factorization of $a_u$}
Proposition \ref{gendecompo} reduces the proof of Theorem\refer{result1newth} to showing that 
for all $\ell$, $V^{(\ell)}\in \mathcal{A}$, which, in view of Remarks \ref{remm}, \ref{rem2}  will also imply that the number of (non-zero) profiles in \eqref{decompNgen}
is bounded by $\ds \frac{\|u_0\|^2_{L^2(\R)}} {4 \pi}\cdotp$ 
To prove this property, we will need the following structural result for $a_{U_n}$.

\begin{proposition}
\label{gendecomposcatcoef}
{\sl With the notations of Proposition\refer{gendecompo}, we have for all $\ds 0<\delta<\frac{\pi}2\virgp$  
\begin{equation}\label{str0}
\limsup\limits_{n\rightarrow\infty}|a_{U_n}^{(4)}(\lambda)-\prod\limits_{\ell=1}^La_{V^{(\ell)}}^{(4)}(\lambda)|\stackrel{L\to\infty}\longrightarrow 0\, ,\end{equation} 
uniformly with respect to $\lambda\in \Gamma_\delta, \, \im (\lambda^2)\geq \delta$.}
\end{proposition}
\begin{proof}
Setting $\ds U_n^{L}(y)=\sum_{\ell=1}^{L} V^{(\ell)}(y-y^{(\ell)}_n)$ and applying \eqref{eq7}, we get for  all $L\geq 1$ and all~$n$ sufficiently large, 
\begin{equation}\label{str1}
\big|a_{U_n}^{(4)}(\lambda)-a_{U_n^{L}}^{(4)}(\lambda)\big|\lesssim_{\|u_0\|_{L^2}, \delta} \|{\rm
r}_n^{L}\|^{\frac 1 2 }_{L^4(\R)} .\end{equation}
Furthermore, \refeq{eq1} and \eqref{inter1} ensure that for all $L$ and all $n$, 
\begin{equation}\label{str2}
\big|a_{U_n^L}^{(4)}(\lambda)-\prod\limits_{\ell=1}^La_{V^{(\ell)}}^{(4)}(\lambda)\big|\lesssim_{\|u_0\|_{L^2}, \delta, L}
\sum\limits_{{1\leq\ell, \ell'\leq L\atop\ell\neq\ell'}}\|T_{V^{(\ell)}(\cdot-y_n^\ell)}(\lambda)T_{V^{(\ell')}(\cdot-y_n^{\ell'})}(\lambda)\|.
\end{equation}
Observing that, for all $f, g$ in $L^2(\R)$ and all $\lam$ in $\C_{++}$, there holds 
$$\|T_f(\lam)T_g(\lam)\|^2\leq C\frac{|\lam|^4}{{(\im (\lam^2))^2}}\int\limits_{\R^2} e^{-2\im(\lam^2)|x-y|} |f(x)|^2|g(y)|^2 dxdy \, ,
$$
we  readily gather that 
\begin{equation}\label{str3}
\big\|T_{V^{(\ell)}(\cdot-y_n^{(\ell)})}(\lambda)T_{V^{(\ell')}(\cdot-y_n^{(\ell')})}(\lambda)\big\|\longrightarrow 0 \quad {\rm as}\,\, 
|y_n^{(\ell)}-y_n^{(\ell')}|\rightarrow \infty,
\end{equation}
uniformly with respect to $\lambda\in \Gamma_\delta, \, \im(\lam^2)\geq \delta$.
Invoking \eqref{str1}, \eqref{str2}, and taking into account\refeq{decompNgen1}, we get \eqref{str0}.
\end{proof}

 \medbreak

As a corollary of the above proposition, we obtain\footnote{using the notations of page \pageref{defimpfunct}} : 
  \begin{cor} 
    \label{coruse1}
{\sl There exists a positive constant $C_{\|u_0\|_{L^2}}$ such that
     \begin{equation}\label{cor0}
    \limsup\limits_{n\rightarrow\infty}   \big|\sum_{\ell=1}^{L} \varphi_{V^{(\ell)}}(\rho)+  \varphi_{0, {\rm
r}_n^{L} }(\rho)\big|  \stackrel{L\to\infty}\longrightarrow 0\end{equation}
for any  $\rho\geq C_{\|u_0\|_{L^2}}$. }\end{cor}   
\begin{proof}
 For any $\rho\geq C_{\|u_0\|_{L^2}}$,  with a suitable constant $C_{\|u_0\|_{L^2}}$,
 we can write   $$\varphi_{U_n}(\rho)= \im  \big(\ln a_{U_n}^{(4)}(\sqrt{i\rho})\big) + \varphi_{0, U_n}(\rho) \, .$$
Invoking Proposition \ref{gendecomposcatcoef} together with \eqref{eq5} and Remark  \ref{rem2}, we infer  that for any $\rho\geq C_{\|u_0\|_{L^2}}$,
\begin{equation}
\label{cor1}
\limsup\limits_{n\rightarrow\infty}\big|\ln a_{U_n}^{(4)}(\sqrt{i\rho})-\sum\limits_{l=1}^L\ln a_{V^{(\ell)}}^{(4)}(\sqrt{i\rho})\big|\stackrel{L\to\infty}\longrightarrow 0\, .\end{equation} 
Furthermore, it follows from \eqref{ortogonalth} that for all $L\geq 1$ and all $ \rho>0$, 
\begin{equation}\label{cor2}
\varphi_{0, U_n}(\rho)=\frac12\int_\R\frac{p^2|\hat U_n(p)|^2}{p^2+4\rho^2}dp=\sum\limits_{\ell=1}^L
 \varphi_{0, V^{(\ell)}}(\rho)+\varphi_{0, {\rm r}_n^L}(\rho) +o(1), \quad  n\rightarrow \infty.
\end{equation}
Observe also that
due to the scaling property \eqref{sca} and the bound \eqref{eq6}, we have
\begin{equation}\label{cor3}
\varphi_{U_n}(\rho) \stackrel{n\to\infty}\longrightarrow 0, \quad \forall \rho>0,
\end{equation}
which together with \eqref{cor1} and  \eqref{cor2} gives \eqref{cor0}.
\end{proof}

\medbreak

\subsubsection{End  of the proof of  the rigidity type theorem}\label{secondprop}
Here we complete the proof of Theorem \ref{result1newth}.  This will be done by combining  Lemma \ref{studyfirstpartimaginary} and Corollary \ref{coruse1}. In order   to apply Lemma \ref{studyfirstpartimaginary}, we first need to check that:

\begin{lemma}
\label{lemstep01}
{\sl For each profile $V^{(\ell)}$ involved in the decomposition \refeq{decompNgen}, the spectral coefficient~$a_{V^{(\ell)}}$  does not vanish on~$\C_{++}$.}
\end{lemma}

\begin{proof}
We  proceed by contradiction, assuming that    there exist $\ell_0 \geq 1$, $\lam_0\in \C_{++}$ and  $\psi_{0} \in H^1(\R)$   such that~$\|\psi_{0} \|_{L^2( \R)}=1$ and 
\begin{equation} \label{asymp}  
L_{V^{(\ell_0)}}(\lambda_0)\psi_0=0.
\end{equation}
Then  we have
\begin{equation} \label{syststudy}\
L_{U_n}(\lambda_0)\psi_{0}(\cdot-y^{(\ell_0)}_n)
= \cR_n(y)\, ,\end{equation}
where 
$$
\cR_n (y)=    -i   \lam_0 \left(
\begin{array}{ccccccccc}
0 &\ds \sumetage{\ell \neq \ell_0} {1 \leq  \ell \leq L} V^{(\ell)}(y-y^{(\ell)}_n) +{\rm
r}_n^{L}(y) \\
\overline {\ds \sumetage{\ell \neq \ell_0} {1 \leq  \ell \leq L} V^{(\ell)}(y-y^{(\ell)}_n) +{\rm
r}_n^{L}(y)} &0
\end{array}
\right)\psi_{0}(y-y^{(\ell_0)}_n).$$
The scaling property\refeq{sca} and the estimate  \eqref{eq5} 
 ensure that  $ \ds |a_{U_n}(\lam_0)|\geq \frac 1 2 \virgp$   for $n$ large enough.  It then follows from Proposition\refer{resest}   that  the operator~$L_{U_n}(\lambda_0) $  is invertible 
 and
 $$\|L^{-1} _{U_n}(\lambda_0)\| \leq C_{\lambda_0, \|u_0\|_{L^2}},$$  
 which implies that
 \begin{equation}
  \label{controlpsi0}\|\psi_{0} \|_{L^2( \R)} \leq  C_{\lambda_0, \|u_0\|_{L^2}} \|\cR_n \|_{L^2( \R)} \, .\end{equation}
Consider $\cR_n$.
The orthogonality condition between the cores ensures that, for all~$\ell \neq \ell_0$,  there holds
$$\|V^{(\ell)}(\cdot-y^{(\ell)}_n) \, \psi_{0}(\cdot -y^{(\ell_0)}_n)\|_{L^2( \R)} \stackrel{n \to  + \infty} \longrightarrow 0\, ,$$
 which together with the fact that, for all $2< p< \infty$,  
 $\ds  \limsup_{n\to\infty}\;\|{\rm
r}_n^{L}\|_{L^p(\R)}\stackrel{L\to\infty}\longrightarrow 0 $    allows us to conclude that
\begin{equation} \label{l2normsyststudy} \|\cR_n\|_{L^2( \R)} \stackrel{n \to  + \infty} \longrightarrow 0\, .\end{equation}
Combining \refeq{controlpsi0}, \eqref{l2normsyststudy} and taking into account the fact that $\|\psi_{0} \|_{L^2( \R)} =1$,  we get  a contradiction.  
\end{proof}

\medbreak

We are now in position to finish   the proof of Theorem \ref{result1newth}. From Lemmas \ref{studyfirstpartimaginary} and \ref{lemstep01},  we have for all $\ell\geq 1$,
$$\varphi_{V^{(\ell)}}(\rho)\geq 0, \quad \forall \rho>0.$$
Recalling that 
$$\varphi_{0, {\rm r}_n^L}(\rho)=\frac12\int_\R\frac{p^2|\hat {\rm r}_n^L(p)|^2}{p^2+4\rho^2}dp\geq 0,$$
we deduce from Corollary \ref{coruse1} that for all $\rho\geq C_{\|u_0\|_{L^2}}$,
\begin{equation}\label{nullim}
\varphi_{V^{(\ell)}}(\rho)= 0, \quad \forall \,\ell\geq 1,
\end{equation}
and 
\begin{equation}
\limsup\limits_{n\rightarrow\infty}   \varphi_{0, {\rm
r}_n^{L} }(\rho) \stackrel{L\to\infty}\longrightarrow 0.\end{equation}
In view of Lemma \ref{studyfirstpartimaginary}, Identity \eqref{nullim}
implies that $a_{V^{(\ell)}}\equiv 1$ on $\C_{++}$ for all $\ell\geq 1$.  Accordingly to Remarks\refer{remm} and\refer{rem2}, this ensures
that the number of non zero profiles 
in the decomposition \eqref{decompNgen} is finite and bounded by  $\ds  \frac{\|u_0\|^2_{L^2(\R)}} {4 \pi}\cdotp$   Denoting this number by $L_0$ and setting 
${\rm r}_n={\rm r}_n^{L_0} $, we get   
$$\lim_{n\to\infty}\;\|{\rm
r}_n\|_{L^{p}(\R)}= 0,$$
for all $2<p<\infty$, which concludes the proof of Theorem \ref{result1newth}.

\subsubsection{A key result} 
Our aim now is to show that the profile decomposition given by Theorem\refer{result1newth} is in contradiction with the conservation of $a_u$. To this end,  we will need the following result that ensures the closeness of the functions $a_{U_n}$ and $a_{r_n}$.   \begin{proposition} 
 \label{preturbationspectrumgen}
{\sl Let   $(V^{(\ell)})_{1 \leq \ell \leq L}$ be a finite family of functions in $\mathcal{A}$. For $\underline{\varphi}=(\varphi_\ell)$, $\underline{y}=(y_\ell)$  in~$\R^L$,  we denote $\ds u_{\underline{\varphi}, \underline{y}}(y)= \sum_{\ell=1}^{L} e^{i \varphi_\ell} V^{(\ell)}(y-y_\ell)$.
 Then,   for all~$\ds 0<\delta<\frac\pi2$ and  all~$ m>0$, we have, 
  \begin{equation} 
\label{detyest3imp} a_{u_{\underline{\varphi},\underline{y}}+{ r}} (\lam) - a_{{ r}} (\lam) \longrightarrow 0\, ,
\end{equation} as $\ds \min_{\ell \neq \ell'}| y_\ell - y_{\ell'}| \rightarrow \infty$    and  $\|{ r}\|_{L^4}\rightarrow 0$,
 ${ r}\in L^2(\R) \cap L^4(\R)$ with $\|{ r}\|_{L^2} \leq m$, uniformly with respect to  $\lam\in\Gamma_\delta$ and $\underline{\varphi}\in \R^L$.} 
\end{proposition} 

\begin{proof} In order to establish the result, we shall consider separately the cases~$|\lam|\ll1$ and~$|\lam|\gtrsim 1$, reducing the proof of Proposition \ref{preturbationspectrumgen} to the two   following lemmas:
\begin{lemma} 
\label{lem1}
{\sl Under the assumptions of Proposition \refer{preturbationspectrumgen}, for all $\ds 0<\delta<\frac\pi2$ and  all $ m>0$, 
we have
$$a_{u_{\underline{\varphi},\underline{y}}+r} (\lam) -a_{{r}}(\lambda)\underset{\lambda \to 0,\,\lam\in \Gamma_\delta}{\longrightarrow} 0\,$$
uniformly with respect to $\underline{\varphi}, \,\underline{y} \in \R^L$ and
${ r}\in L^2(\R)$ satisfying   $\|{ r}\|^2_{L^2(\R)} \leq m$.}
\end{lemma}
\medbreak

\begin{lemma} 
\label{lem2infty}
{\sl For all $\ds 0<\delta<\frac\pi2$ and all $\alpha>0$, we have   
$$ a_{u_{\underline{\varphi},\underline{y}}+{r}} (\lam) -a_{{ r}}(\lam)
  \longrightarrow 0\, ,$$  as $ \ds \min_{\ell \neq \ell'} | y_\ell - y_{\ell'}| \to \infty $ and 
  $\|{ r}\|_{L^4(\R)}\to 0$ with $\|{ r}\|^2_{L^2(\R)} \leq m$,  uniformly with respect to~$\underline{\varphi}$  in~$\R^L$ and~$\lam$ in~$\Gamma_\delta \cap \{ |\lam|\geq \alpha\}$.  }
\end{lemma}

\medskip We start with the proof of the first lemma:
\begin{proof}[Proof of Lemma {\rm\ref{lem1}}] 

\medskip  First, recall that in view of   \refeq{trace2useful}, \refeq{link24} and \refeq{eq1}, we have 
$$   \big|a_{u_{\underline{\varphi},\underline{y}}+{ r}} (\lam) - a_{{r}} (\lam) \big|  
\lesssim_{\delta, m}
\big|\tr T_{u_{\underline{\varphi},\underline{y}}+{r}}^2(\lam)-\tr T_{{r}}^2(\lam)\big| +  
 \big|a^{(4)}_{u_{\underline{\varphi},\underline{y}}+{ r}}(\lam)- a^{(4)}_{{ r}}(\lam) \big|   \, .$$
Then,  note that it  follows from\refeq{trace2useful} that,  for all $\lam \in \Gamma_\delta$, all $\underline{\varphi}, \underline{y}\in \R^L$
and ${ r}\in L^2(\R)$ with~$\|{ r}\|^2_{L^2(\R)} \leq m$, there holds  
$$\big|\tr T_{u_{\underline{\varphi},\underline{y}}+{ r}}^2(\lam)-\tr T_{{ r}}^2(\lam)\big| \lesssim_{\delta, m}
|\lam|\sum_{\ell=1}^L \big\|(p+2\lam^2)^{-1/2}\hat V^{(\ell)} \big\|_{L^2(\R)}\,
\underset{\lam \to 0, \, \lam\in \Gamma_\delta}{\longrightarrow} 0.$$

\smallskip  To estimate 
$a^{(4)}_{u_{\underline{\varphi},\underline{y}}+{ r}}(\lam)- a^{(4)}_{{ r}}(\lam)$,
we use 
\refeq{detyest6}. Approximating the potentials $V^{(\ell)}$ by functions of $L^2(\R)\cap L^{1}(\R)$ and using that
$$\|T_f(\lam)T_g(\lam)\|\leq C\frac{|\lam|^2}{\sqrt{\im (\lam^2)}}\|f\|_{L^1}\|g\|_{L^2}, \quad \forall \lam \in \C_{++}, \, f\in L^1(\R), \, g\in L^2(\R),$$
one easily checks  that
$$
\|T_{u_{\underline{\varphi}, \underline{y}}}^2(\lam)\|+\|T_{u_{\underline{\varphi}, \underline{y}} }(\lam) T_{{ r}}(\lam)\|\underset{\lam\in \Gamma_\delta, \,\lam \to 0}{\longrightarrow} 0,
$$
uniformly with respect to $\underline{\varphi}, \,\underline{y} \in \R^L$ and
${r}\in L^2(\R)$, $\|{ r}\|^2_{L^2(\R)} \leq m$.
Therefore, applying \refeq{detyest6}, we get
$$
a^{(4)}_{u_{\underline{\varphi},\underline{y}}+{ r}}(\lam)- a^{(4)}_{{ r}}(\lam)
\underset{\lam\in \Gamma_\delta, \,\lam \to 0}{\longrightarrow} 0,
$$
uniformly with respect to $\underline{\varphi}, \,\underline{y} \in \R^L$ and
${r}\in L^2(\R)$ with $\|{r}\|^2_{L^2(\R)} \leq m$,
which  achieves the   proof of the lemma.

 \end{proof} 
 
 \medbreak

\noindent{\it Proof of Lemma {\rm\ref{lem2infty}}}.
We start by observing 
that according to\refeq{trace2useful},\refeq{link24},  \refeq{eq1} and the fact that $V^{(\ell)} \in \mathcal{A}$, we have 
for all $\lam\in \Gamma_\delta$, all $\underline{\varphi}, \underline{y}\in\R^L$,
and ${ r}\in L^2(\R)$ with $\|{ r}\|^2_{L^2(\R)} \leq m$  
\begin{equation*}\begin{split}
 |a_{u_{\underline{\varphi},\underline{y}}+{ r}} (\lam) - a_{{ r}} (\lam)|  
\lesssim_{\delta, m} \,\,&\Big|\underbrace{a_{{ r}}^{(4)}(\lam)-1}_{\cJ_1}\Big |+\Big|\underbrace{
a^{(4)}_{u_{\underline{\varphi},\underline{y}}+{ r}} (\lam)-\prod^L_{\ell=1}a^{(4)}_{V^{(\ell)}} (\lam)}_{\cJ_2}\Big|+
\\&+
\Big|\underbrace{\tr \left(T_{u_{\underline{\varphi},\underline{y}}+{ r}}^2(\lam)-T_{{ r}}^2(\lam)-\sum^L_{\ell=1}T_{V^{(\ell)}}^2(\lam)\right)}_{\cJ_3}\Big|
 \, \cdotp\end{split}\end{equation*}
 By virtue of Estimate \eqref{eq5}, we have 
 $$\big|\cJ_1\big|\lesssim_{\delta, \alpha, m}\|r\|^4_{L^4(\R)},$$
 for all $\lam\in \Gamma_\delta$, $|\lam|\geq \alpha$ and all $r\in L^2(\R)\cap L^4(\R)$ with $\|{r}\|^2_{L^2(\R)} \leq m$.
 
 \smallskip We next address  $\cJ_2$. Arguing as  in the proof of Proposition \ref{gendecomposcatcoef},
  we get
  $$\big|\cJ_2\big|\lesssim_{\delta, \alpha, m} \|r\|_{L^4(\R)}^{\frac 1 2 }+\sum\limits_{{1\leq\ell, \ell'\leq L\atop\ell\neq\ell'}}\|T_{e^{i\varphi_\ell}V^{(\ell)}(\cdot-y_\ell)}(\lambda)T_{e^{i\varphi_\ell'}V^{(\ell')}(\cdot-y_{\ell'})}(\lambda)\|,$$
  with 
  $$\sum\limits_{{1\leq\ell, \ell'\leq L\atop\ell\neq\ell'}}\|T_{e^{i\varphi_\ell}V^{(\ell)}(\cdot-y_\ell)}(\lambda)T_{e^{i\varphi_\ell'}V^{(\ell')}(\cdot-y_{\ell'})}(\lambda)\|\underset{ \min\limits_{\ell \neq \ell'}| y_\ell - _{\ell'}| \to \infty}{\longrightarrow} 0,$$
  uniformly with respect to $\lam\in \Gamma_\delta$, $|\lam|\geq \alpha$,  and $\underline{\varphi}\in \R^L$.
  \medskip   

 In order to end the proof of the lemma, it remains to estimate $\cJ_3$:
 $$\cJ_3=
 2\sumetage  { 1\leq   \ell, \ell' \leq  L}{  \ell \neq  \ell' } \tr \bigl(T_{e^{i \varphi_\ell}V^{(\ell)}(\cdot -y_\ell)}(\lam)\, T_{e^{i \varphi_{\ell'}}V^{(\ell')}(\cdot -y_{\ell'})}(\lam)\bigr) + 2 \sum_{ 1\leq   \ell \leq  L}\tr\bigl(T_{e^{i \varphi_\ell}V^{(\ell)}(\cdot -y_\ell)}(\lam)\, T_{{r} }(\lam)\bigr).$$
 Clearly,
 $$\big | \tr(T_f(\lam)T_g(\lam))\big |\leq 2|\lam|^2\int\limits_{\R^2} e^{-2\im(\lam^2)|x-y|} |f(x)||g(y)| dxdy, \quad \forall \lam \in \C_{++}, \, f\in L^2(\R), \, g\in L^2(\R),$$
 which readily implies that
 $$\sumetage  { 1\leq   \ell, \ell' \leq  L}{  \ell \neq  \ell' } \tr \bigl(T_{e^{i \varphi_\ell}V^{(\ell)}(\cdot -y_\ell)}(\lam)\, T_{e^{i \varphi_{\ell'}}V^{(\ell')}(\cdot -y_{\ell'})}(\lam)\bigr) \underset{ \min\limits_{\ell \neq \ell'}| y_\ell - y_{\ell'}| \to \infty}{\longrightarrow} 0,$$
  uniformly with respect to $\lam\in \Gamma_\delta$ and $\underline{\varphi}\in \R^L$, and
  $$ \sum_{ 1\leq   \ell \leq  L}\tr\bigl(T_{e^{i \varphi_\ell}V^{(\ell)}(\cdot -y_\ell)}(\lam)\, T_{{r} }(\lam)\bigr)
\underset{\|r\|_{L^2}\leq m, \|r\|_{L^4}\to 0}{\longrightarrow} 0,$$
  uniformly with respect to $\lam\in \Gamma_\delta$ and $\underline{y},\underline{\varphi}\in \R^L$.
 This completes the proof of the lemma and therefore, the proof of Proposition  \ref{preturbationspectrumgen}  as well.\end{proof} 
 
\medbreak 

 \subsection{End of the proof of   Theorem\refer{Mainth}} \label{End}

Assume that $u_0\in H^{\frac 1 2}(\R)$, with $\|u_0\|^2_{L^2(\R)}\geq 4\pi$, is such that the solution $u\in C([0, T^*[, H^{\frac12}(\R))$ of the DNLS equation with 
initial data $u(0)=u_0$ verifies~$\sup\limits_{0\leq t< T^*}\|u(t)\|_{H^{\frac 1 2}(\R)}=+\infty$. For a sequence $(t_n)_{n \in \N} \subset [0, T^*[$ such that~$\|u(t_n)\|_{H^{\frac 1 2}(\R)}\to +\infty$, as $n$ goes to infinity, 
set as above   $$\ds U_n(x)= \frac 1 {\sqrt{\mu_n}} u(t_n, \frac x { \mu_n}) \with \mu_n=\|u(t_n)\|^2_{\dot H^{\frac 1 2}(\R)}\cdotp $$ 
Let us  next take a sequence of functions~$(u^{(k)}_0)_{k \in \N}$ in  $\cS_{\rm reg} (\R)$ that converges in $H^{\frac12}$ to $u_0$.
Denoting by $u^{(k)}(t)$ the solution of the DNLS equation with initial data $u^{(k)}(0)=u^{(k)}_0$, we have
$$ u^{(k)}  \stackrel{k \to  + \infty} \longrightarrow u  \, \, \mbox{in} \, \,  C([0, T], H^{\frac 1 2}(\R)), \, \, \forall \,\,T< T^*\, .$$
Consequently, 
the functions~$\ds U_n^{(k)}\eqdefa   \frac 1 {\sqrt{\mu_n}} u^{(k)}\Big(t_n, \frac \cdot { \mu_n}\Big)$ belong to $\cS_{\rm reg} (\R)$    and satisfy for any integer~$n$\begin{equation}
\label{convbcexfin}U_n^{(k)}  \stackrel{k \to  + \infty} \longrightarrow U_n  \, \, \mbox{in} \, \,  H^{\frac 1 2}(\R) \, .\end{equation}

 Since for any fixed~$0<\theta<\pi$,  the function $\tilde a_{u_0}(\zeta) $ admits at most a  finite number of zeros  in the  angles 
 $\{\zeta\in \C: \, \theta\leq \arg \zeta<\pi\}$,   there exists $\ds \frac\pi2<\theta_0<\pi$ such that   \begin{equation}
\label{invscalnew} \tilde a_{u_0} (\zeta) \neq 0\, , \, \, \forall  \zeta \, ,  \, \, \theta_0 \leq  \arg \zeta  < \pi \,.\end{equation}
We denote by $\zeta_{n, j}^{(k)}$, $j=1, \dots, I_{k}$ the zeros of $\tilde a_{U_n^{(k)}}(\zeta)$ in the angle $\{\zeta\in \C: \, \theta_0\leq \arg \zeta<\pi\}$ (because of the scaling property \refeq{sca}, the number of zeros $I_k$ does not depend on $n$).

\medskip The key ingredient in the proof of Theorem\refer{Mainth} is given by  the following lemma,   the proof of which is postponed to  the end of this paragraph. \begin{lemma}
\label{claimresult2fin}
{\sl  With the above notations, we have
\begin{equation}\label{number}
\liminf\limits_{k\rightarrow \infty}I_{k}\geq 1.\end{equation}
Moreover, numbering the zeros $\zeta_{n, j}^{(k)}$  so that  
$$ {\rm Re} (\zeta_{n, 1}^{(k)})= \min_{1 \leq j \leq I_{k}} {\rm Re} (\zeta_{n, j}^{(k)})  \,,$$ one has \begin{equation}
\label{propzerosfinalbis} \liminf_{n \to  \infty} \liminf_{k \to  \infty} {\rm Re} (\zeta_{n, 1}^{(k)})= c_0 < 0   \, . \end{equation} } 
 \end{lemma}

 \medbreak
 Admitting  for a while  Lemma\refer{claimresult2fin}, let us achieve the proof of Theorem\refer{Mainth}.  
 Accordingly to the lemma, for all $k$ sufficiently large and all $n\in\N$,  the function  $\tilde a_{U_n^{(k)}}$ has at least one zero in~$\{\zeta\in \C: \, \theta_0\leq \arg \zeta<\pi\}$. 
   By\refeq{sca},    for all  $j=1,\cdots, I_{k}$,     
   $$\tilde a_{u^{(k)}_0}(\zeta^{(k)}_{n, j}   \mu_n)= 0  \,  . $$
In view of Lemma\refer{lemmcont}, this implies that for all $k$ sufficiently large and all $n$,
$$\re \zeta^{(k)}_{n, 1} \geq  - 
\frac{C}{\mu_n}  \virgp $$ 
for  some positive constant $C$ depending only  on $\|u_0\|_{H^{\frac 1 2} (\R)}$.
Since 
$\mu_n  \stackrel{n \to +\infty}\longrightarrow +\infty$, this contradicts the property
\eqref{propzerosfinalbis}.

 \bigskip Thus, to complete the proof of  the theorem, we need to establish Lemma\refer{claimresult2fin}.  To this end, let us start by  observing  that in view of Theorem \ref{result1newth}, we have 
\begin{equation}
\label{decbc}  U_n^{(k)}(y)= \sum_{\ell=1}^{L_0}    V^{(\ell)}(y-y^{(\ell)}_n)  +  {
r}^{(k)}_n(y), 
\quad {
r}^{(k)}_n= {
\rm r}_n+ U_n^{(k)}-U_n , \end{equation}
with
\begin{equation}
\label{decbc1}
\min\limits_{\ell\neq\ell'}|y^{(\ell)}_n-y^{(\ell')}_n| \stackrel{n \to +\infty}\longrightarrow +\infty\quad {\rm and}\quad
\lim\limits_{k\rightarrow +\infty} \|r^{(k)}_n\|_{L^p(\R)}\stackrel{n \to +\infty}\longrightarrow 0\, ,\end{equation}  
for all $2<p<\infty$.

 \smallskip Combining  \refeq{invscalnew} with Corollary \ref{coruse}
  and taking into account the stability estimate\refeq{eqstab} and the scaling property  \refeq{sca}, we infer that there exist $C\geq 1$ 
   and $K\in \N$ such that for all $k\geq K$, all $n\in \N$ and all $\zeta\in \C_+$
 with $\arg \zeta=\theta_0$ we have
\begin{equation}
\label{estnew}   \frac1C \leq \Big|\frac  1 {\tilde  a_{U^{(k)}_n}(\zeta)}  \Big| \leq  C. 
\end{equation}
Invoking Proposition\refer{preturbationspectrumgen}, we deduce   that 
\begin{equation}
\label{justifnew}    
\limsup\limits_{k\rightarrow +\infty}\sup\limits_{{\zeta\in \C_+\atop \arg \zeta=\theta_0}}
 \Big|1- \frac  {\tilde   a_{{
r}^{(k)}_n}(\zeta)} {\tilde  a_{U^{(k)}_n}(\zeta)}  \Big| \stackrel{n \to +\infty}\longrightarrow 0
\, \cdotp \end{equation}
  It follows then that, for all $n$ and $k$ large enough, $\tilde  a_{r_n^{(k)}}$ does not vanish on the ray $e^{i\theta_0}\R_+^*$
  and therefore, we can apply
   Lemmas \refer{intrel}, \ref{intrelrootsnew}, which gives
$$I_{k}\geq \frac{1}{2i\pi }\int\limits^{+ \infty \, e^{i\theta_0}}_0 \left(\frac {\tilde a'_{U_n^{(k)}}(s)}  {\tilde  a_{U_n^{(k)}}(s)} -
\frac {\tilde a'_{r_n^{(k)}}(s)}  {\tilde  a_{r_n^{(k)}}(s)}\right)
ds+\frac1{4\pi}\left({\|U_n^{(k)}\|^ 2_{L^2(\R)}}-\|r_n^{(k)}\|^ 2_{L^2(\R)}\right).
$$
In view of \eqref{ortogonalth} and \eqref{justifnew}, we have respectively
$$\lim\limits_{k\rightarrow +\infty}\left({\|U_n^{(k)}\|^ 2_{L^2(\R)}}-\|r_n^{(k)}\|^ 2_{L^2(\R)}\right)\stackrel{n \to +\infty}\longrightarrow
 \sum_{\ell=1}^{L_0} \|V^{(\ell)}\|_{L^2(\R)}^2,$$
 and $$\limsup_{k\rightarrow+\infty}\left|\int^{+ \infty \, e^{i\theta_0}}_0 \left(\frac {\tilde a'_{U_n^{(k)}}(s)}  {\tilde  a_{U_n^{(k)}}(s)} -
\frac {\tilde a'_{r_n^{(k)}}(s)}  {\tilde  a_{r_n^{(k)}}(s)}\right)ds\right| \stackrel{n \to +\infty}\longrightarrow 0,$$
which shows that
 $$\liminf\limits_{k\rightarrow \infty}I_{k}\geq  \frac1{4\pi}\sum_{\ell=1}^{L_0} \|V^{(\ell)}\|_{L^2(\R)}^2\geq 1.$$

 \medskip  To establish   \eqref{propzerosfinalbis}, we argue by contradiction assuming that
 \begin{equation}
 \label{propzerosfinalbis1}
 \liminf_{n \to  \infty} \liminf_{k \to  \infty} \re \zeta_{n, 1}^{(k)}\geq 0.
 \end{equation}
 Since $\ds \frac\pi2<\theta_0<\pi$, this means that  \begin{equation}
 \label{propzerosfinalbis1}
 \lim_{n \to  \infty} \limsup_{k \to  \infty}\max\limits_{j=1, \dots I_{k}} |\zeta_{n, j}^{(k)}|=0.
 \end{equation}
 Observe also that since $\tilde a_u$ does not vanish in the angle $\{\zeta\in \C: \, \theta_0\leq \arg \zeta<\pi\}$, it follows from  the stability estimate \eqref{eqstab} and  Corollary \ref{coruse} that 
   \begin{equation} \label{estzeros2}  \max\limits_{j=1, \dots, I_k}\frac {\im \zeta^{(k)}_{n, j} }{|\re \zeta_{n,j}^{(k)}|} \stackrel{k\to\infty}\longrightarrow 0 \, \cdotp\end{equation}

 \bigskip 
 We shall now eliminate one by one all the zeros $\zeta_{n, j}^{(k)}$ by applying  recursively the B\"acklund transformation
 that we have introduced  in Paragraph\refer{defBacklund transformation}.  Namely, consider
   the family of Schwartz class functions $(U^{(k)}_{n, j})_{j=0, \cdots,I_{k}}$   defined by 
 \begin{eqnarray*} U^{(k)}_{n, 0} &=&  U^{(k)}_{n}\, , \\  U^{(k)}_{n, j}&=& \cB_{\lam_{n, j}^{(k)}} (\eta_{n, j}^{(k)}) U_{n, j-1}^{(k)} \, , \, \, j=1,\cdots, I_{k} \, ,\end{eqnarray*} 
 or explicitly, 
 $$
  U^{(k)}_{n, j}=G^2_{\lam_{n, j}^{(k)}} (\eta_{n,j}^{(k)})U^{(k)}_{n, j-1}-G_{\lam_{n, j}^{(k)}} (\eta_{n,j}^{(k)})S_{\lam_{n, j}^{(k)}} (\eta_{n,j}^{(k)}),\quad j=1,\cdots, I_{k},$$
 where 
 $\lam_{n, j}^{(k)}=\sqrt{\zeta_{n, j}^{(k)}}\in \C_{++}$, and $\eta_{n, j}^{(k)}\in L^2(\R, \C^2)\setminus\{0\}$, 
 $L_{U_{n, j-1}^{(k)}}(\lam_{n, j}^{(k)})\eta_{n, j}^{(k)}=0$.  It then follows from\refeq{mass1} and  \eqref{backtransfrela}     that 
 \begin{equation}\label{mass2}
 \|U^{(k)}_{n, j}\|_{L^2(\R)}^2= \|U^{(k)}_{n}\|_{L^2(\R)}^2- 4 \sum^j_{i=1}  { \arg} (\zeta_{n, i}^{(k)}) \, 
  \end{equation}
 (see also Appendix \ref{proofderG}, Remark \ref{remap}).
 Since $|G_{\lam_{n, j}^{(k)}} (\eta_{n,j}^{(k)})|=1$,  the above relation ensures that
 there exists a positive constant $C$ such that for all~$n, k, j$,  there holds  
 \begin{equation}\label{estS}
 \|S_{\lam_{n, j}^{(k)}} (\eta_{n,j}^{(k)})\|_{L^2(\R)}\leq C.\end{equation}
 
 Combining this bound with the estimate\refeq{SLinfty},  we readily gather by an obvious induction that,  for all $2< p<  \infty$, $$ \|U^{(k)}_{n, j}\|_{L^p(\R))}\leq  \|U^{(k)}_{n}\|_{L^p(\R))}+C  \sum^j_{i=1} \big(\im\lambda_{n, i}^{(k)}) \big)^{1-\frac 2 p}\, .$$
 Let us now consider the functions $\cW^{(k)}_{n} \eqdefa U^{(k)}_{n, I_{k}}$. Clearly,
  \begin{equation} \begin{aligned}
\label{newdef} \cW^{(k)}_{n} & =  G^2_{\lam_{n, I_{k}}^{(k)}} (\eta_{n, I_{k}}^{(k)}) \cdots G^2_{\lam_{n, 1}^{(k)}} (\eta_{n,1}^{(k)})U^{(k)}_{n}
\\ &  \qquad \qquad  \qquad \qquad - \sum^{I_{k}} _{j=1} G^2_{\lam_{n, {I_k}}^{(k)}} (\eta_{n, I_k}^{(k)}) 
\cdots G^2_{\lam_{n, j+1}^{(k)}} (\eta_{n,j+1}^{(k)})G_{\lam_{n, j}^{(k)}}(\eta_{n,j}^{(k)})\cS_{\lam_{n, j}^{(k)}}  (\eta_{n,j}^{(k)}). \end{aligned}
\end{equation}
We claim that there exists a family   $(\beta^{(k)}_{n, \ell})_{1\leq  \ell \leq L_0}$ of complex numbers of modulus $1$ such that,  for all $2<p< \infty$,  there holds\begin{equation}
\label{newdec}  \cW^{(k)}_{n}= \sum_{\ell=1}^{L_0}   \beta^{(k)}_{n, \ell} V^{(\ell)}(\cdot-y^{(\ell)}_n) +\cR^{(k)}_{n}\, , \quad   \limsup_{k \to  \infty} \|\cR^{(k)}_{n}\|_{L^p(\R)}\stackrel{n \to  + \infty} \longrightarrow 0\, .\end{equation}
Indeed, setting  $$\beta^{(k)}_{n, \ell}= \underbrace{\big(G^2_{\lam_{n, I_{k}}^{(k)}} (\eta_{n, I_{k}}^{(k)}) \cdots G^2_{\lam_{n, 1}^{(k)}} (\eta_{n,1}^{(k)})\big)}_{\cG^{(k)}_{n}}(y^{(\ell)}_n)\, ,$$
we obviously obtain a family of complex numbers of modulus $1$. Then, taking advantage of\refeq{decbc}, we deduce that   
$$\cR^{(k)}_{n}=\cR^{(k)}_{n, 1}+\cR^{(k)}_{n, 2}+\cR^{(k)}_{n, 3} , $$
with
\begin{eqnarray*} \cR^{(k)}_{n, 1}(y) &=& \cG^{(k)}_{n}(y) \, {
r}^{(k)}_n(y)\, , \\ \cR^{(k)}_{n, 2}(y) &=& - \sum^{I_{k}} _{j=1} G^2_{\lam_{n, {I_k}}^{(k)}} (\eta_{n, I_k}^{(k)}) 
\cdots G^2_{\lam_{n, j+1}^{(k)}} (\eta_{n,j+1}^{(k)})G_{\lam_{n, j}^{(k)}}(\eta_{n,j}^{(k)})\cS_{\lam_{n, j}^{(k)}}  (\eta_{n,j}^{(k)})
\\ \cR^{(k)}_{n, 3}(y) &=& \sum^{L_0} _{\ell=1} (\cG^{(k)}_{n}(y)-\cG^{(k)}_{n}(y^{(\ell)}_n))V^{(\ell)}(y-y^{(\ell)}_n)\,   . \end{eqnarray*}
Since $|\cG^{(k)}_{n}|=1$,  \refeq{decbc1} ensures that,
for all $2<p< \infty$,  $$\ds  \lim_{k \to  \infty} \|\cR^{(k)}_{n, 1}\|_{L^{p}(\R)}\stackrel{n \to  + \infty} \longrightarrow  0\,  . $$
For $R^{(k)}_{n, 2}$, we have
$$\|\cR^{(k)}_{n, 2}\|_{L^p(\R)}\leq \sum^{I_{k}} _{j=1} \|\cS_{\lam_{n, j}^{(k)}}  (\eta_{n,j}^{(k)})\|_{L^p(\R)}\lesssim   I_{k} \max_{1 \leq j \leq I_{k}} \big(\im \lam^{(k)}_{n, j}\big)^{1-\frac 2 p}\,  . $$
Note that by \eqref{controbound},  $I_{k}$ is bounded independently of $k$.
Therefore, taking into account \eqref{propzerosfinalbis1}, we deduce from the above estimate
 that $$\ds  \limsup_{k \to  \infty} \|\cR^{(k)}_{n, 2}\|_{L^p(\R)}\stackrel{n \to  + \infty} \longrightarrow 0\,  . $$
Finally, in order to investigate  the term $\cR^{(k)}_{n, 3}$, we shall make use of the pointwise estimate\refeq{derG} which gives:  \begin{equation*} 
\big |\frac d {dy}\cG^{(k)}_{n} (y)\big| \leq 16\sum^{I_{k}} _{j=1}\Big( (\im \lam^{(k)}_{n, j})^2+ \im( \lam^{(k)}_{n, j})|U^{(k)}_{n, j-1} (y)|\Big)   \, .\end{equation*}
Clearly, for all $j$,  we have
$$|U^{(k)}_{n, j} (y)|\leq |U^{(k)}_{n} (y)|+ \sum^{I_{k}} _{j=1}\big|\cS_{\lam_{n, j}^{(k)}}  (\eta_{n,j}^{(k)}(y))\big|\leq  |U^{(k)}_{n} (y)|+ 4\sum^{I_{k}} _{j=1}\im (\lam_{n, j}^{(k)}).$$
Therefore, combining the two last inequalities, we obtain: 
\begin{equation} \label{pointwiseest}
\big |\frac d {dy}\cG^{(k)}_{n} (y)\big| \lesssim  \Big(\sum^{I_{k}} _{j=1} |\lam^{(k)}_{n, j}|\Big)^2+ |U^{(k)}_{n} (y)|\sum^{I_{k}} _{j=1} |\lam^{(k)}_{n, j}|  \, ,\end{equation}
which implies that, for all $1\leq  \ell \leq L_0$, 
$$ \big|\cG^{(k)}_{n}(y+y_n^{(\ell)})-\cG^{(k)}_{n}(y^{(\ell)}_n) \big| \lesssim  \Big(\sum^{I_{k}} _{j=1} |\lam^{(k)}_{n, j}|\Big)^2 \, |y| + \sum^{I_{k}} _{j=1} |\lam^{(k)}_{n, j}| \, |y|^{\frac 1 2} \|U^{(k)}_{n}\|_{L^2(\R)}\, .$$
In view of \eqref{propzerosfinalbis1}, this  ensures that, for all $1\leq  \ell \leq L_0$  and all $R>0$, $$\limsup\limits_{k\rightarrow +\infty} \sup\limits_{y\in [-R, R]}\big|\cG^{(k)}_{n}(y+y_n^{(\ell)})-\cG^{(k)}_{n}(y^{(\ell)}_n) \big| \stackrel{n \to  + \infty} \longrightarrow  0.$$
Since $\big |G^{(k)}_{n} (y)\big|= 1$, this allows us to conclude that
$$\limsup_{k \to  \infty} \|\cR^{(k)}_{n, 3}\|_{L^p(\R)}\stackrel{n \to  + \infty} \longrightarrow 0\, ,$$
for all $2\leq p<\infty$,  which ends the proof of\refeq{newdec}. 

\medskip Let us now consider  $\tilde a_{\cW^{(k)}_{n}} $. By \eqref{backtransfrela}, it has the following form 
\begin{equation}
\label{est010}   
\tilde a_{\cW^{(k)}_{n}}(\zeta)=\tilde a_{U^{(k)}_{n}}(\zeta)\prod\limits_{j=1}^{I_{k}}\frac{\zeta-\overline\zeta^{(k)}_{n, j}}
{\zeta-\zeta^{(k)}_{n, j}}\,. \end{equation}
Due to \eqref{estzeros2}, we have  
\begin{equation}
\label{est10}    
\sup\limits_{{\zeta\in \C_+\atop \arg \zeta=\theta_0}}
 \Big|1- \prod\limits_{j=1}^{I_{k}}\frac{\zeta-\overline\zeta^{(k)}_{n, j}}
{\zeta-\zeta^{(k)}_{n, j}}\Big|\stackrel{k \to  + \infty} \longrightarrow 0,
  \end{equation}
which in view of  \eqref{est010} and \eqref{estnew} implies that for all $k$ sufficiently large, all $n\in \N$ and all $\zeta\in \C_+$
 with $\arg \zeta=\theta_0$,
we have
\begin{equation}
\label{estnew1}   \frac1{2C} \leq \Big|\frac  1 {\tilde a_{\cW_n^{(k)}}(\zeta)}  \Big| \leq  2C. 
\end{equation}
This bound together with \eqref{newdec} and \eqref{decbc1} allows us to apply Proposition\refer{preturbationspectrumgen} and
Lemmas \refer{intrel}, \ref{intrelrootsnew} to the sequence $(\cW^{(k)}_{n})$,
repeating the argument we have used above to prove
 \eqref{number}. Taking into account that by construction, the function $\tilde a_{\cW_n^{(k)}}(\zeta) $ has no zero in the angle
$\{\zeta\in \C_+: \,\, \theta_0 \leq \arg \zeta < \pi  \big\}$, we obtain
 $$0\geq 1+ \frac{1}{2i\pi }\int\limits^{+ \infty \, e^{i\theta_0}}_0 \left(\frac {\tilde a'_{\cW_n^{(k)}}(s)}  {\tilde  a_{\cW_n^{(k)}}(s)} -
\frac {\tilde a'_{\cR_n^{(k)}}(s)}  {\tilde  a_{\cR_n^{(k)}}(s)}\right)
ds+\frac1{4\pi}\left({\|\cW_n^{(k)}\|^ 2_{L^2(\R)}}- \sum_{\ell=1}^{L_0} \|V^{(\ell)}\|_{L^2(\R)}^2-\|\cR_n^{(k)}\|^ 2_{L^2(\R)}\right),
$$
with
$$\limsup_{k\rightarrow+\infty}\left|\int^{+ \infty \, e^{i\theta_0}}_0 \left(\frac {\tilde a'_{\cW_n^{(k)}}(s)}  {\tilde  a_{\cW_n^{(k)}}(s)} -
\frac {\tilde a'_{\cR_n^{(k)}}(s)}  {\tilde  a_{\cR_n^{(k)}}(s)}\right)ds\right| \stackrel{n \to +\infty}\longrightarrow 0,$$
and
$$\limsup_{k\rightarrow+\infty}\left|{\|\cW_n^{(k)}\|^ 2_{L^2(\R)}}- \sum_{\ell=1}^{L_0} \|V^{(\ell)}\|_{L^2(\R)}^2-\|\cR_n^{(k)}\|^ 2_{L^2(\R)}\right|\stackrel{n \to +\infty}\longrightarrow 0,$$
which gives a contradiction and therefore, concludes the proof of lemma.
 Thus   the proof of Theorem\refer{Mainth} is  achieved.

\appendix 

\section{Regularized Determinants}\label{basicdeterminants}
In this appendix, we review the basic properties of  the regularized determinants  ${\rm det}_n({\rm I}-A)$ for~$A$ in~$\mathscr{C}_n$,   the set of bounded operators\footnote{For our purpose, we focus here on   $L^2$-framework, but all the results are available on separable Hilbert spaces.}~$A$ on~$L^2(\R)$ such that   $|A|^n$ is of trace-class, endowed with the norm $\|A\|_n \eqdefa \big[{ \rm Tr}  \big(|A|^n\big)\big] ^{\frac 1 n}$. We refer to the monograph  of Simon\ccite{Simon} and the references therein for further details and the proofs.

\smallskip To introduce the  regularized determinants, let us  start by defining,  for any bounded operator~$A$ on $L^2(\R)$,
$$R_n (A)={\rm I}- ({\rm I}-A) \exp \Big(\sum^{n-1}_{k=1} \frac {A^k} k\Big)  \cdotp$$ 
Clearly, 
\begin{equation}
\label{RA0}
R_n(A)=A^nh_n(A),
\end{equation}
where $h_n$ is an entire function\footnote{Explicitly, $h_n(z)=z^{-n}\left(1-(1-z)e^{\sum^{n-1}_{k=1} \frac {z^k} k}\right)$.}on $\C$.
This shows  that $R_n (A)$ belongs to~$\mathscr{C}_1$ if $A$ is in~$\mathscr{C}_n$,
which justifies the following definition:
\begin{definition}  
\label{defdety}
{\sl For any  operator  $A$ in~$\mathscr{C}_n$, $n \geq 2$, we define 
\begin{equation} 
\label{defdetn}{\rm det}_n ({\rm I}-A) \eqdefa {\rm det}({\rm I}-R_n (A)) = {\rm det}\Big(({\rm I}-A) \exp \Big(\sum^{n-1}_{k=1} \frac {A^k} k\Big)\Big)\cdotp\end{equation}
}
\end{definition}
\begin{remark} 
\label{defdetyrk}
{\sl
The above formula deserves some comments:
\begin{itemize} 
   \item For any $n\geq 2$, if $A$ is in $\mathscr{C}_{n-1}$,  then\footnote{with the convention $\det_1(I-A)=\det(I-A)$}
\begin{equation} 
\label{dety1}{\rm det}_n ({\rm I}-A)= {\rm det}_{n-1} ({\rm I}-A)  \exp \Big(  \frac {{\rm Tr}(A^{n-1})} {n-1}\Big) \cdotp\end{equation} 
 \item Note also that, for all $A$ in $\mathscr{C}_{n}$ such that $\|A\|<1$ (or more generally $\|A^p\|<1$, for some~$p$), one has: \begin{equation} 
\label{dety2}{\rm det}_n ({\rm I}-A)=\exp \Big( - {\rm Tr} \sum^{\infty}_{k=n} \frac {A^k} k\Big) \cdotp\end{equation}
 \end{itemize} }
\end{remark} 

\medbreak

In the following proposition, we  summarize some useful properties of the regularized determinants:
\begin{proposition}
\label{defdetyrk}
{\sl With the previous notations, for any $n\geq 1$ there exists a positive constant~$C_n$ such that the following holds.
\begin{enumerate} 
\item  For all~$A \in \mathscr{C}_n$,
\begin{equation} 
\label{boundedest}\big|{\rm det}_n ({\rm I}-A)\big| \leq \exp \big(C_n \|A\|_n^n\big), \end{equation}   
\begin{equation} 
\label{detyest3}|{\rm det}_n ({\rm I}-A)- 1| \leq C_n\|A^n\|_1 \exp \big(C_n \|A\|_n ^n\big). \end{equation}
 \item For all~$A, B$ in $\mathscr{C}_n$, 
 \begin{equation} 
\label{detyest4}|{\rm det}_n ({\rm I}-A)- {\rm det}_n ({\rm I}-B))| \leq \|A-B\|_n \exp \big(  C_n \big(\|A\|_n+ \|B\|_n+  1\big)^n\big) .
\end{equation}
\item Let ~$A \in \mathscr{C}_n$. Then~${\rm I}-A$ is invertible if and only if ~${\rm det}_n ({\rm I}-A)\neq 0$, and furthermore, one has 
 \begin{equation} 
\label{detyest5}\| ({\rm I}-A)^{-1}\| \leq \frac {C_n} {|{\rm det}_n ({\rm I}-A)|} \exp\big(  C_n \ \|A\|^n_n\big). \end{equation}
 .\end{enumerate} }
\end{proposition}

\medbreak

We conclude  this appendix by  the following result that we have used in the proof of Theorem\refer{Mainth}.  
\begin{proposition} 
\label{propresproduct}
{\sl  For any $m\geq 0$,  there exists a  positive constant $C_m$ such that for all $A, B$ in~$\mathscr{C}_{2}$ satisfying  $\|A\|_2+ \|B\|_2 \leq m$ we have 
\begin{equation} \label{inter1}
|{\rm det}_4(I-A-B)-{\rm det}_4(I-A){\rm det}_4(I-B)|\leq C_m(\|AB\|+\|BA\|),
\end{equation}
and
\begin{equation} 
\label{detyest6}|{\rm det}_4 ({\rm I}-A-B)- {\rm det}_4 ({\rm I}-B)| \leq C_m \big(\|A^2\|+ \|AB\|\big).\end{equation}

}
\end{proposition} 
\begin{proof}  
In view of  the definition\refeq{defdety}, we have
$$\|R_4(A+B)\|_1+\|R_4(A)\|_1+\|R_4(B)\|_1\leq C_m,$$
and
$$ \|R_4(A+B)-R_4(A)-R_4(B)+R_4(A)R_4(B)\|_1\leq C_m\big(\|AB\|+\|BA\|\big).$$
Combining these inequalities with the bound \eqref{detyest4} and invoking the identity
$$
\det(I+C)\det(I+D)=\det(I+C+D+CD), \quad \forall \, C, D\in \mathscr{C}_{1},
$$
we get \eqref{inter1}.

\medskip To prove Estimate \eqref{detyest6}, we start by observing that $R_4(A+B)$ can be written in the form
$$R_4(A+B)=R_4(B)+f(B)A+\mathcal{G}(A,B), $$
where 
$f(z)= z^3 h_4(z)$ and the remainder $\mathcal{G}(A,B)$ admits the bound:
$$\|\mathcal{G}(A,B)\|_1\leq C_m \big(\|A^2\|+ \|AB\|\big).$$
Furthermore, we have
$$\|f(B)AR_4(B)\|_1\leq C_m  \|AB\|.$$
Therefore, 
$$\|R_4(A+B)-R_4(B)-f(B)A+f(B)AR_4(B)\|_1\leq C_m \big(\|A^2\|+ \|AB\|\big),$$
which implies
\begin{equation}
\label{pdet1}
|\det(I-R_4(A+B))-\det(I-R_4(B))\det(I-f(B)A)|\leq C_m \big(\|A^2\|+ \|AB\|\big).
\end{equation}
To conclude the proof of  \eqref{detyest6},   it remains to observe that
$$|\det(I-f(B)A)-1|=|\det(I-Af(B))-1|\leq C_m\|AB\|.$$

 \end{proof} 

\section{Proof of the estimate\refeq{derG}}\label{proofderG}
According to\refeq{defS}, we  have    \begin{equation} 
\label{calculder} \frac d {dx}G_{\lam} (\eta)  = \big(\lam^2-\overline \lam^2\big) \frac {\big[|\eta_1|^2\frac d {dx} (|\eta_2|^2)-|\eta_2|^2\frac d {dx} (|\eta_1|^2)\big]} {d_\lam^2(\eta)}\,  \cdotp    \end{equation} 
Since, in view of\refeq{sp}, 
\begin{eqnarray*}  \partial_x \eta_1&= & -i \lam^2 \eta_1+ \lam u \eta_2 \\ \partial_x \eta_2 &= &  \,   \,  \, i \lam^2 \eta_2 -  \lam \overline u \eta_1   ,\end{eqnarray*}
we infer that 
$$ |\eta_1|^2\frac d {dx} (|\eta_2|^2)= 2  |\eta_1|^2 {\rm Re}\big(   i \lam^2 |\eta_2|^2 -  \lam \overline u \eta_1  \overline \eta_2 \big),$$
and
$$ |\eta_2|^2\frac d {dx} (|\eta_1|^2)= 2  |\eta_2|^2 {\rm Re}\big(  - i \lam^2 |\eta_1|^2 +  \lam   u \eta_2  \overline \eta_1 \big) .$$
Therefore,
$$\begin{aligned}  |\eta_1|^2\frac d {dx} (|\eta_2|^2)-|\eta_2|^2\frac d {dx} (|\eta_1|^2)&= 2i(\lam^2-\bar \lam^2)|\eta_1|^2  |\eta_2|^2
\\  & - u\bar \eta_1\eta_2d_{\bar \lam}(\eta)-\bar u\eta_1\bar \eta_2d_\lam(\eta),
 \end{aligned}$$
which implies that 
\begin{equation}\label{AA}
  \frac d {dx}G_{\lam} (\eta) = -\frac{i}2G_\lam(\eta)\left(\left|\cS_\lam(\eta)\right|^2-uG_\lam(\eta)\overline{\cS_\lam(\eta)}
- \overline{uG_\lam(\eta)}\cS_\lam(\eta)\right).
\end{equation}
Taking into account that $|G_\lam(\eta)| = 1$ and
$ |\cS_{\lam} (\eta)|  \leq 4\im \lam$, we obtain
 $$\Big|\frac d {dx}G_{\lam} (\eta) (x)\Big| \leq 8 (\im\lam)^2 + 4\im(\lam) |u(x)|.$$
 
\begin{remark} \label{remap}
  {\sl  Assuming that $\eta\in L^2(\R,\C^2)\setminus\{0\}$, one easily deduces from \eqref{AA} that
  $$\|\cB_\lam(\eta)u\|_{L^2(\R)}^2=\|u\|_{L^2(\R)}^2-8\arg \lam.$$ }
  \end{remark}


\begin{thebibliography}{50}

 \bibitem{Ab} M.-J.  Ablowitz and H. Segur, Solitons and the inverse scattering transform,  {\it Siam Studies in Applied Mathematics},  1981.


\bibitem{G} M. Agueh, Sharp Gagliardo-Niremberg inequalities and mass transport theory,  {\it Journal of Dynamics and Differential Equations},  {\bf 18},  pages  1069-1093, 2006.

 

 
   \bibitem{bealscoifman}
R.~Beals and R.-R. Coifman, Scattering and inverse scattering for first order systems, {\it Communications on Pure and Applied Mathematics},
 {\bf 37},  pages 39-90, 1984.
 



\bibitem{Biagioni} H. Biagioni and F. Linares,   Ill-posedness for the derivative Schr\"odinger  and generalized Benjamin-Ono equations, {\it Transactions of the American Mathematical Society},  {\bf 353},  pages  3649-3659, 2001.  

  


\bibitem{Champeaux} D. Champeaux-Laveder, T. Passot, P.-L. Sulem, Remarks on the parallel propagation of small-amplitude dispersive, {\it Alfven waves Nonlinear Processes in Geophysics},  {\bf 6},  pages  169-178, 1999.

\bibitem{CKSTT} J. Colliander, M. Keel, G. Staffilani, H. Takaoka and T.  Tao, Refined global well-posedness result for Schr\"odinger  equations with derivative, {\it SIAM Journal on Mathematical Analysis},  {\bf 34},  pages  64-86, 2002.

\bibitem{Colin}  M. Colin and M. Ohta, Stability of solitary waves for derivative nonlinear Schr\"odinger equation, {\it Annales de l'Institut Henri  Poincar\'e, Analyse nonlin\'eaire},  {\bf 23},  pages  753-764, 2006.

\bibitem{DK}  P. Deift and R. Killip, On the absolutely continuous spectrum of one-dimensional Schr\"odinger operators with square summable potentials,     {\it Communications in Mathematical Physics},  {\bf 203},  pages   341-347, 1999.


\bibitem{pgerard1}
P. G\'erard,  Description du d\'efaut de compacit\'e de l'injection de Sobolev, {\it  ESAIM Control, Optimisation and Calculus of Variations}, {\bf 3},  pages  213-233, 1998.



\bibitem{Guo1} B. Guo and  Yaping  Wu, Orbital stability of solitary waves for the nonlinear derivative Schr\"odinger  equation, {\it Journal of Differential Equations},  {\bf 123},  pages  35-55, 1994.   

\bibitem{AG} A. Gr\"unrock, Bi- and trilinear Schr\"odinger estimates in one space dimension with applications to cubic NLS and DNLS,  {\it International Mathematics Research Notices},  {\bf 41},  pages  2525-2558, 2005.   


\bibitem{Guo} Z. Guo and Yifei  Wu,  Global well-posedness for the derivative nonlinear Schr\"odinger equation in $H^{\frac 1 2}$,  {\it Discrete and Continuous Dynamical Systems},  {\bf 1},  pages  64-86, 2017.   
 


\bibitem{HO}  N. Hayashi and T. Ozawa,  On the derivative nonlinear Schr\"odinger equation, {\it Physica D},  {\bf 55},  pages  14-36, 1992. 

\bibitem{SHW}   J. He, L. Wang and S.Xu, The Darboux transformation of the derivative nonlinear Schr\"odinger equation, {\it Communications in Mathematical Physics},   {\bf 44},  pages  567-576, 2011. 




\bibitem{Sulem1} R. Jenkins, J. Liu, P.  Perry and  C. Sulem, Global existence for the derivative nonlinear Schr\"odinger equation with arbitrary spectral singularities, to appear in {\it Analysis and PDE}.
 
\bibitem{Sulem} R. Jenkins, J. Liu, P.  Perry and  C. Sulem, The derivative nonlinear Schr\"odinger equation: global well-posedness and soliton resolution,   {\it Quarterly Journal of Pure and Applied Mathematics}, {\bf 78}, pages 33-73, 2020.

\bibitem{Sulem0} R. Jenkins, J. Liu, P.  Perry and  C. Sulem,  Global well-posedness for the derivative nonlinear Schr\"odinger equation,   {\it  Communications on  Partial Differential Equations}, {\bf 43}, pages 1151-1195, 2018.  



\bibitem{KN} D.-J. Kaup and A.-C. Newell,  An exact solution for a  derivative nonlinear Schr\"odinger equation, {\it Journal of Mathematics Physics},  {\bf 19},  pages  798-801, 1978.



 \bibitem{killip0} R. Killip, M.  Visan, KdV is well-posed in $H^{-1}$,  {\it  Annals of Mathematics},  {\bf 190},    pages 249-305, 2019.

 \bibitem{killip} R. Killip, M.  Visan and X. Zhang,   Low regularity conservation laws for integrable PDE, {\it Geometric and  Functional  Analysis},   {\bf 28},  pages  1062-1090, 2018.  
 
  \bibitem{ks} F.  Klaus and R.     Schippa,  A priori estimates for the derivative nonlinear Schr\"odinger equation,  {\it arXiv:2007.13161}. 
 
 \bibitem{KT} H. Koch and D. Tataru, Conserved energies for the cubic nonlinear Schr\"odinger equation in one dimension, {\it  Duke Mathematical  Journal},  {\bf 167},  pages  3207-3313, 2018.  
 

 
  \bibitem{KW}  S.  Kwon and Y.  Wu, Orbital stability of solitary waves for derivative nonlinear Schr\"odinger equation,  {\it Journal d'Analyse Math\'ematique},   {\bf 135},  pages  473-486, 2018. 
 
  \bibitem{Lee}  J.-H. Lee,   Global solvability of the derivative nonlinear Schr\"odinger  equation,  {\it Transactions of the American Mathematical Society},   {\bf 314},  pages  107-118, 1989.   
 
 
 \bibitem{MOMT}  W. Mio, T. Ogino, K. Minami, S. Takeda, Modified nonlinear Schr\"odinger for Alfven waves propagating along
the magnetic field in cold plasma, {\it Journal of the Physical Society of Japan},   {\bf 41},  pages  265-271, 1976. 
 
  \bibitem{mjohus} E. Mjolhus,   On the modulational instability of hydromagnetic waves parallel to the magnetic
field,  {\it Journal of Plasma Physics},  {\bf 16},  pages  321-334, 1976.

 \bibitem{Moses} J. Moses, B.-A. Malomed, and F.-W. Wise, Self-steepening of ultrashort optical pulses without self-phase-
modulation, {\it Physical Review A},  {\bf 76},  2007.  


 \bibitem{Ozawa0}  T. Ozawa, On the nonlinear Schr\"odinger equations of derivative type,  {\it  Indiana University Mathematics Journal},  {\bf 45},  pages  137-163, 1996.

 

 
\bibitem{PSS} D.-E. Pelinovsky, A. Saalmann and Y. Shimabukuro, The derivative NLS equation: global existence with solitons,  {\it Dynamics of PDE},  {\bf 14},  pages  271-294, 2017. 

\bibitem{PSS2} D.-E. Pelinovsky  and Y. Shimabukuro, Existence of global solutions to the derivative NLS equation with the inverse scattering method, {\it International Mathematics Research Notices},   pages 5663-5728, 2017. 


\bibitem{Simon} B. Simon,  Trace ideals and their applications, Second edition. { \it Mathematical surveys and Monographs}, {\bf 120},  American Mathematical Society, Providence, 2005.



 \bibitem{HT} H. Takaoka,   Well-posedness for the one-dimensional nonlinear Schr\"odinger equation with the derivative nonlinearity,  {\it Advances in Differential Equations},   {\bf 4},  pages  561-580, 1999.  
 
 \bibitem{HT2} H. Takaoka,  Global well-posedness for Schr\"odinger  equations with derivative in a nonlinear term and data in
low-order Sobolev spaces, {\it Electronic Journal of Differential Equations},   {\bf 42},  pages  1-23, 2001.   

  \bibitem{T2}  T. Tsuchida and M. Wadati,  Complete integrability of derivative nonlinear Schr\"odinger-type equations,  {\it Inverse Problems},   {\bf 15},  pages  1363-1373, 1999.  
  

 
   \bibitem{Weinstein}   M.-I. Weinstein, Nonlinear Schr\"odinger equations and sharp interpolation estimates, {\it Communications in Mathematical Physics},   {\bf 87},  pages  567-576, 1983. 


 
\bibitem{W}  Y.  Wu, Global well-posedness on the derivative nonlinear Schr\"odinger equation revisited,  {\it Analysis and PDE},  {\bf 8},  pages  1101-1112, 2015.  
\end{thebibliography}
\end{document}